\documentclass[11pt]{amsart}

\usepackage[utf8]{inputenc}
\usepackage[T1]{fontenc}
\usepackage[english]{babel}

\usepackage[margin=7em]{geometry}
\usepackage{needspace}

\usepackage{amssymb, amsfonts}

\usepackage{stmaryrd}
\usepackage[mathscr]{euscript}

\usepackage{enumerate}

\usepackage{float}
\usepackage{booktabs}

\usepackage{url}

\numberwithin{equation}{section}

\newtheorem{theoremcounter}{theoremcounter}[section]

\newtheorem{corollary}[theoremcounter]{Corollary}
\newtheorem{definition}[theoremcounter]{Definition}

\newtheorem{lemma}[theoremcounter]{Lemma}
\newtheorem{proposition}[theoremcounter]{Proposition}
\newtheorem{remark}[theoremcounter]{Remark}

\newtheorem{theorem}[theoremcounter]{Theorem}

\newcommand{\nbd}{\nobreakdash-\hspace{0pt}}

\newcommand{\cal}{\ensuremath{\mathcal}}
\newcommand{\bboard}{\ensuremath{\mathbb}}

\newcommand{\cE}{\ensuremath{\cal{E}}}
\newcommand{\cF}{\ensuremath{\cal{F}}}

\newcommand{\cM}{\ensuremath{\cal{M}}}

\newcommand{\cP}{\ensuremath{\cal{P}}}

\newcommand{\cR}{\ensuremath{\cal{R}}}
\newcommand{\cS}{\ensuremath{\cal{S}}}

\newcommand{\bbF}{\ensuremath{\bboard F}}

\newcommand{\bbL}{\ensuremath{\bboard L}}

\newcommand{\rmF}{\ensuremath{\mathrm{F}}}

\newcommand{\rmJ}{\ensuremath{\mathrm{J}}}

\newcommand{\rmM}{\ensuremath{\mathrm{M}}}

\newcommand{\rmt}{\ensuremath{\mathrm{t}}}

\newcommand{\amid}{\ensuremath{\mathop{\mid}}}

\newcommand{\NN}{\ensuremath{\mathbb{N}}}
\newcommand{\ZZ}{\ensuremath{\mathbb{Z}}}
\newcommand{\QQ}{\ensuremath{\mathbb{Q}}}

\newcommand{\CC}{\ensuremath{\mathbb{C}}}

\newcommand{\isdiv}{\amid}
\newcommand{\nisdiv}{\ensuremath{\mathop{\nmid}}}

\renewcommand{\pmod}[1]{\ensuremath{\;(\mathrm{mod}\, #1)}}

\newcommand{\ra}{\ensuremath{\rightarrow}}

\newcommand{\GL}[1]{\ensuremath{\mathrm{GL}_{#1}}}
\newcommand{\SL}[1]{\ensuremath{\mathrm{SL}_{#1}}}

\newcommand{\Sp}[1]{\ensuremath{\mathrm{Sp}_{#1}}}

\newcommand{\rT}{\ensuremath{{}^\rmt\hspace{-0.08em}}}

\newcommand{\tr}{\ensuremath{\mathrm{tr}}}

\newcommand{\rk}{\ensuremath{\mathop{\mathrm{rk}}}}

\newcommand{\slashdiv}{\ensuremath{\mathop{/}}}

\newcommand{\HS}{\mathbb{H}}

\newcommand{\wtd}{\widetilde}
\newcommand{\ov}{\overline}

\newcommand{\HH}{\ensuremath{\mathbb{H}}}

\begin{document}

\title[The structure of Siegel modular forms modulo $p$ and $U(p)$ congruences]{The structure of Siegel modular forms modulo $p$\\ and $U(p)$ congruences}

\author{Martin Raum}
\address{ETH Zurich, Mathematics Dept., CH-8092, Zürich, Switzerland}
\email{martin.raum@math.ethz.ch}
\urladdr{http://www.raum-brothers.eu/martin/}

\author{Olav K. Richter}
\address{Department of Mathematics\\University of North Texas\\ Denton, TX 76203\\USA}
\email{richter@unt.edu}

\thanks{The author was partially supported by the ETH Zurich Postdoctoral Fellowship Program and by the Marie Curie Actions for People COFUND Program.  The second author was partially supported by Simons Foundation Grant $\#200765$}



\begin{abstract}
We determine the ring structure of Siegel modular forms of degree $g$ modulo a prime $p$, extending Nagaoka's result in the case of degree~$g=2$.  We characterize $U(p)$ congruences of Jacobi forms and Siegel modular forms, and surprisingly find different behaviors of Siegel modular forms of even and odd degrees.
\end{abstract}
\maketitle

\section{Introduction and statement of main results}
\label{sec:introduction}


Serre~\cite{Serre-p-adic} and Swinnerton-Dyer's~\cite{SwD-l-adic} theory of modular forms modulo a prime $p$ has impacted different research areas.  In particular, it has been the basis of beautiful results on congruences of Fourier series coefficients of modular forms.  Ono \cite{Ono} gives a good overview of the subject, and he highlights several applications of congruences that involve Atkin's $U$-operator (see also Ahlgren and Ono~\cite{Ahl-Ono-Comp05}, Elkies, Ono, and Yang~\cite{E-O-Y-IMRN05}, and Guerzhoy~\cite{Gue-PAMS2007}).  Siegel modular forms modulo a prime~$p$ have conjectural connections to special values of $L$-functions (as predicted by Harder~\cite{Harder-1-2-3} for degree~$2$ and Katsurada~\cite{Kats-MathZ08} for general degree), and one wishes for a better understanding of such Siegel modular forms.  Nagaoka~\cite{Na-MathZ00, Na-MathZ05} establishes the ring structure of Siegel modular forms of degree $2$ modulo a prime~$p$.  B\"ocherer and Nagaoka~\cite{Bo-Na-MathAnn07, Bo-Na-Abh10} and Ichikawa~\cite{Ichi-MathAnn08} provide tools to study the ring of Siegel modular forms of degree $g$ modulo a prime $p$, but the structure of that ring has not been determined if $g>2$.  In this paper, we fill this gap by providing a structure theorem for the ring of Siegel modular forms of arbitrary degree modulo a prime $p$.  Let us introduce necessary notation to state our result.

Throughout, $k, g \ge 1$ are integers, and $p \geq 5$ is a prime.  We write $\bbF_p$ for the field~$\ZZ \slashdiv p \ZZ$ and $\ZZ_{(p)}$ for the localization of $\ZZ$ at the principal ideal~$(p)$, and we call it the ring of $p$-integral rationals.  Let $\rmM^{(g)}_k (\ZZ_{(p)})$ be the vector space of Siegel modular forms of degree~$g$, weight~$k$, and with $p$-integral rational coefficients.  Let $\Phi_1, \ldots, \Phi_n$ be generators of $\rmM^{(g)}_\bullet(\ZZ_{(p)})$, and consider the abstract isomorphism
\begin{gather}
\label{eq:siegel-reduction-isomorphism}
  \rmM^{(g)}_\bullet(\ZZ_{(p)})
\cong
  \ZZ_{(p)}[x_1, \ldots, x_n] \slashdiv C
\text{,}
\qquad
  \Phi_i \mapsfrom x_i
\text{,}
\end{gather}
where $C \subset \ZZ_{(p)}[x_1, \ldots, x_n]$ is an ideal.  The right hand side of~\eqref{eq:siegel-reduction-isomorphism} carries a $\ZZ$-grading which is induced by the weight grading on the left hand side.  This grading is inherited by
\begin{gather*}
  \bbF_p[x_1, \ldots, x_n] \slashdiv C
=
  \big( \ZZ_{(p)}[x_1, \ldots, x_n] \slashdiv C \big) \otimes \bbF_p
\text{.}
\end{gather*}
We always write $B \in \ZZ_{(p)}[x_1, \ldots, x_n] \slashdiv C$ for the polynomial that corresponds under the isomorphism~\eqref{eq:siegel-reduction-isomorphism} to the Siegel modular form $\Psi_{p-1} \equiv 1 \pmod{p}$ constructed by B\"ocherer and Nagaoka~\cite{Bo-Na-MathAnn07}.
Our first main theorem extends the known ring structures of Siegel modular forms modulo $p$ of degree $g=1$ (\cite{SwD-l-adic}) and $g=2$ (\cite{Na-MathZ00}) to arbitrary degree $g$.

\begin{theorem}
\label{thm:structure-of-siegel-mod-p-rings}
Let $p \ge g + 3$.  We have
\begin{gather*}
  \rmM^{(g)}_\bullet(\bbF_p)
\cong
  \big ( \rmM^{(g)}_\bullet(\ZZ_{(p)}) \otimes \bbF_p \big) \slashdiv \langle 1 - \Psi_{p-1}\rangle
\cong
  \bbF_p [x_1, \ldots, x_n] \slashdiv \big(C + \langle 1 - B\rangle \big)
\text{,}
\end{gather*}
where $\langle 1 - B\rangle$ and $\langle 1 - \Psi_{p - 1}\rangle$ are the principal ideals generated by $1 - B$ and $1 - \Psi_{p -1}$, respectively.
\end{theorem}
\noindent The proof of Theorem~\ref{thm:structure-of-siegel-mod-p-rings} is not a mere generalization of the results in \cite{Na-MathZ00}, and relies on first developing properties of Jacobi forms of higher degree modulo $p$.

There are congruences between Siegel modular forms of different weights, and it is desirable to determine the smallest weight in which such a congruence takes place.  This leads to the notion of the filtration~$\omega(\,\cdot\,)$ of Siegel modular forms (for details, see Section~\ref{sec:structureofmodularformsmodp}), which is described in the following corollary.
\begin{corollary}
\label{cor:siegel-characterization-of-maximal-filtration}
Let $p \ge g + 3$.  If $\Phi \in \rmM^{(g)}_k(\ZZ_{(p)})$ with corresponding polynomial~$A$ under~\eqref{eq:siegel-reduction-isomorphism}, then $\omega(\Phi) = k$ if and only if $A$ is not divisible by $B$.
\end{corollary}

Now we turn the attention to another main point of this paper.  Tate's theory of theta cycles (see $\S7$ of \cite{Joch-1982}) yields a criterion for the existence of $U(p)$ congruences of modular forms.  That concept has been applied to Jacobi forms (\cite{crit, heat}) and to Siegel modular forms of degree $2$ (\cite{C-C-R-Siegel}).  In this paper, we extend \cite{C-C-R-Siegel} to the case of general degree $g$.  Specifically, we explore theta cycles of Siegel modular forms of degree $g$, which allows us to determine conditions on the existence of $U(p)$ congruences of such Siegel modular forms.  We state our result after introducing more notation.

For a variable $Z=(z_{ij})$  in the Siegel upper half space of degree $g$ set $\partial_Z:=\left(\frac{1}{2}(1+\delta_{ij})\frac{\partial}{\partial z_{ij}}\right)$.  The generalized theta operator
\begin{gather*}
\label{Theta^g}
\mathbb{D}:=(2\pi i)^{-g}\det \partial_Z
\text{}
\end{gather*}
acts on Fourier series expansions of Siegel modular forms as follows:
\begin{gather*}
  \mathbb{D}\bigg(\sum_{T = \rT T \:\ge\: 0 }\!\!
  c(T)\, e^{2\pi i\,\tr(TZ)}\bigg)
=
  \sum_{T = \rT T \:\geq\: 0 }\!\!
  \det(T) c(T)\, e^{2\pi i\,\tr(TZ)}
\text{,}
\end{gather*}
where $\tr$ denotes the trace, $\rT T$ is the transpose of $T$, and where the sum is over all symmetric, semi-positive definite, and half-integral $g\times g$ matrices with integral diagonal entries.  Moreover, the analog of Atkin's $U$-operator for Siegel modular forms is defined as follows:
\begin{gather*}
  \bigg(
  \sum_{T = \rT T \:\geq\: 0 }\!\!
  c(T)e^{2\pi i\,\tr(TZ)}
  \bigg)\, \bigg| \, U(p)
:=
  \sum_{\substack{T = \rT T \:\geq\: 0 \\ p \isdiv \det T}}\!\!
  c(T) \, e^{2\pi i\,\tr(TZ)}
\text{.}
\end{gather*}
Our second main theorem gives a criterion for the existence of $U(p)$ congruences of Siegel modular forms of degree $g$.
\begin{theorem}
\label{thm:Siegel-U_p}
Let $p \ge \max(k,g + 3)$.  Suppose that $\Phi\in\rmM^{(g)}_k (\ZZ_{(p)})$ has a Fourier-Jacobi coefficient $\Phi_M$ (for details, see Sections~\ref{sec:Siegelpreliminaries}) such that $p\,\nmid \,\det(2M)$ and $\Phi_M\not\equiv 0\pmod{p}$.
\begin{enumerate}[a)]
\item 
Let $g$ be odd.  Then
\begin{gather*}
  \omega\Big(\mathbb{D}^{p+1-k+\frac{g-1}{2}}(\Phi)\Big)
=
  \begin{cases}
    2p+1+g-k\text{,} & \text{if $\Phi \big|\, U(p) \not\equiv 0 \pmod{p}$;}
    \\
    \hphantom{2}p+2+g-k\text{,} & \text{if $\Phi \big|\, U(p) \equiv 0 \pmod{p}$.}
  \end{cases}
\end{gather*}

\item
Let $g$ be even and  $k > \frac{g + 4}{2}$.
\begin{enumerate}[i)]
\item If $p > 2k-g-3$, then $\Phi \, \big | \, U(p) \not\equiv 0 \pmod{p}$.
\item 
   If $p < 2k - g - 3$, then
\begin{gather*}
  \omega\Big(\mathbb{D}^{\frac{3p+g+1}{2}-k}(\Phi)\Big)
=
  \begin{cases}
    3p+1+g-k\text{,} & \text{if $\Phi \big|\, U(p) \not\equiv 0 \pmod{p}$;}
    \\
    2p+2+g-k\text{,} & \text{if $\Phi \big|\, U(p) \equiv 0 \pmod{p}$.}
  \end{cases}
\end{gather*}
\end{enumerate}
\end{enumerate}
\end{theorem}
Theorem~\ref{thm:Siegel-U_p} reduces to Tate's original result (see $\S7$ of~\cite{Joch-1982}) if $g=1$, and to the main result of~\cite{C-C-R-Siegel} if $g=2$.  It is remarkable that the parity of $g$ dictates the criterion.  If $g$ is even, then the result is stronger, since one can exclude $U(p)$ congruences modulo $p$ for almost all primes $p$.  A similar phenomenon occurs in the case of Jacobi forms of higher degree (see Theorem~\ref{thm:Jacobi-U_p}), which we illuminate in Remark~\ref{rem:U_p-explanation}. 

\vspace{1ex}

This paper naturally divides into two parts:  Section~\ref{sec:jacobiforms} is on Jacobi forms and Section~\ref{sec:Siegelmodularforms} is on Siegel modular forms.  We often consider the Fourier-Jacobi expansions of Siegel modular forms and apply our findings from Section~\ref{sec:jacobiforms} to prove results in Section~\ref{sec:Siegelmodularforms}.  In Section~\ref{sec:jacobiformspreliminaries}, we give some background on Jacobi forms.  In Section~\ref{sec:jacobiformsmodp}, we establish properties of Jacobi forms on $\HH \times\CC^l$ modulo $p$, which extend results of Sofer~\cite{Sof-JNT97} and \cite{heat} on the $l=1$ case.  This generalization is somewhat intricate, since (unlike the $l=1$ case) there is no explicit basis for the ring of Jacobi forms of higher degree.  In Section~\ref{sec:expceptional-set-for-jacobi-forms},  we show that the $\rmM^{(1)}_\bullet(\ZZ_{(p)})$\nbd module $\rmJ_{\bullet, M}(\ZZ_{(p)})$ (defined in Section~\ref{sec:expceptional-set-for-jacobi-forms}) is free for all $p \ge 5$.  In Section~\ref{sec:HeatcyclesandUcongruences}, we extend \cite{crit} to Jacobi forms of higher degree.  In particular,  we prove Theorem~\ref{thm:Jacobi-U_p}, which characterizes $U_p$ congruences of Jacobi forms on $\HH \times\CC^l$.  In Section~\ref{sec:Jacobiexamples}, we illustrate Theorem~\ref{thm:Jacobi-U_p} with explicit examples.   In Section~\ref{sec:Siegelpreliminaries}, we give some background on Siegel modular forms.  In Section~\ref{sec:modularformsmodp}, we explore Siegel modular forms modulo $p$.  In Section~\ref{sec:structureofmodularformsmodp}, we study the structure of Siegel modular forms modulo $p$, and we prove Theorem~\ref{thm:structure-of-siegel-mod-p-rings}.  In Section~\ref{sec:ThetacyclesandUcongruences}, we investigate theta cycles of Siegel modular forms.  We prove Proposition~\ref{prop:filtration-of-D-operator}, which describes the relation of the weight filtration and the theta operator.  The proof of Proposition~\ref{prop:filtration-of-D-operator} relies on many results of this paper, and in particular, it depends on Corollary~\ref{cor:siegel-characterization-of-maximal-filtration}, Theorem~\ref{thm:exception-set-is-trivial}, and Proposition~\ref{prop: filtration-of-heat-application}.  Theorem~\ref{thm:Siegel-U_p} follows from Proposition~\ref{prop:filtration-of-D-operator}.  Finally, in Section~\ref{sec:Siegelexample}, we apply Theorem~\ref{thm:Siegel-U_p} in the case of the Schottky form $J_4$, which (up to normalization) is the unique Siegel cusp form of weight $8$ and degree $4$.  We find that $J_4\,\big|\,U(7)\equiv 0\pmod{7}$, while $J_4\,\big|\,U(p)\not\equiv 0\pmod{p}$ for $p=5$ and $p>7$.

\vspace{1ex}

\noindent
{\it Acknowledgments}:
To be entered after the referee's report is received in its final form.

\section{Jacobi forms}
\label{sec:jacobiforms}

\subsection{Preliminaries}
\label{sec:jacobiformspreliminaries}
Throughout, $M$ is a symmetric, positive definite, half-integral $l\times l$ matrix with integral diagonal entries.  Let $\HH\subset\CC$ be the usual complex upper half plane.  Eichler and Zagier \cite{EZ} systematically study Jacobi forms on $\HH\times\CC$, and Ziegler \cite{Zi} introduces Jacobi forms of higher degree.  Set $\Gamma^\rmJ := \SL{2}(\ZZ) \ltimes \ZZ^{2 l}$ and let $\Gamma \subseteq \Gamma^\rmJ$ be a subgroup.  Given a ring $R \subseteq \CC$, let $\rmM^{(1)}_k (R)$ be the space of elliptic modular forms of weight $k$ with coefficients in $R$, and let $\rmJ_{k, M}(\Gamma, R)$ be the space of Jacobi forms of higher degree on $\HS \times \CC^l$ of weight~$k$ and index~$M$ whose Fourier series coefficients at any cusp are contained in~$R$.  If $l=1$, then we always write $\rmJ_{k, m}(\Gamma, R)$ with $m:=M\in\ZZ$.  We suppress $\Gamma$, if $\Gamma = \Gamma^\rmJ$, and $R$, if $R = \CC$.  Recall that $\phi\in\rmJ_{k, M}$ satisfies the transformation law
\begin{multline}
\label{eq:Jacobi transformation law}
  \phi\left( \frac{a\tau+b}{c\tau+d},\,
             \frac{z+\lambda\tau+\mu}{c\tau+d} \right)
\\
=
  (c\tau+d)^{k}\,
  \exp\Big(2\pi i \big( \frac{c\, M[z+\lambda\tau+\mu]}
                              {c\tau+d}
                        - M[\lambda]\tau - 2 \rT \lambda M z
      \big)\Big)\,
  \phi(\tau,z)
\text{,}
\end{multline}
for all $\left[\left(\begin{smallmatrix}a & b\\ c & d\end{smallmatrix}\right),\lambda,\mu\right]\in \Gamma^\rmJ$, where here and throughout the paper $U[V]: = \rT{V} U V$ for matrices $U,V$ of appropriate size.  Furthermore, $\phi$ has a Fourier series expansion of the form
\begin{gather}
\label{eq:Jacobi Fourier expansion}
  \phi(\tau,z)
=
  \sum_{\begin{smallmatrix} n\geq 0,r\in\ZZ^l \\ 4 \det(M)\, n - M^{\#}[r] \geq 0 \\\end{smallmatrix}} \hspace{-1em}
  c(\phi;n, r)\, q^n\zeta^r
\text{,}
\end{gather}
where $q := \exp( 2\pi i\, \tau)$ ($\tau\in\HH$), $\zeta^r := \exp(2 \pi i\, \rT r z)$ ($z \in \CC^l$, $r \in \ZZ^l$), and where $M^{\#}$ is the adjugate of $M$.  If the expansion in~\eqref{eq:Jacobi Fourier expansion} is only over $n$ with $4 \det(M) n - M^{\#}[r] > 0$ then $\phi$ is a Jacobi cusp form of weight~$k$ and index~$M$.  If $\phi: \HS \times \CC^l\rightarrow\CC$ is holomorphic, $\phi$ satisfies (\ref{eq:Jacobi transformation law}), and $\phi$ has a Fourier series expansion as in (\ref{eq:Jacobi Fourier expansion}), but with the difference that $n\gg-\infty$, then $\phi$ is a weakly-holomorphic Jacobi form of weight $k$ and index~$M$.  

\begin{remark}
\label{rm:semi-definite-indices}
Consider the previously excluded case that $M$ is semi-positive definite.   Note that there exists an $U \in \GL{l}(\ZZ)$ such that $M[U] = \left(\begin{smallmatrix} M' & 0 \\ 0 & 0 \end{smallmatrix}\right)$, where $M'$ is a positive definite $l' \times l'$ matrix with $l' = {\rm rank}\, M$.  Writing Jacobi forms~$\phi(\tau, z)$ of index~$M$ in terms of coordinates~$(U z)_i$, allows one to apply standard arguments (see Theorem~1.2 in~\cite{EZ}) to find that $\phi$ is constant with respect to $(U z)_i$ for $l' + 1 \le i \le l$.  Hence the results that we prove for Jacobi forms of positive definite index carry over to Jacobi forms of semi-positive definite index.
\end{remark}

Equation~\eqref{eq:Jacobi transformation law} with $\left(\begin{smallmatrix} a & b \\ c & d \end{smallmatrix}\right) = \left(\begin{smallmatrix} 1 & 0 \\ 0 & 1 \end{smallmatrix}\right)$ yields the so-called theta decomposition, which implies the following isomorphism~(for a nice exposition, see \cite{Sko-Weil}).
\begin{gather}
\label{eq:theta-decomposition}
  \Theta_M :\,
  \rmJ_{k, M}
\longrightarrow
  \rmM^{(1)}_{k - \frac{l}{2}} ({\check \rho}_M)
\text{,}
\end{gather}
where $\rmM^{(1)}_{k - \frac{l}{2}}({\check \rho}_M)$ is the space of vector-valued elliptic modular forms of weight~$k - \frac{l}{2}$ and type~${\check \rho}_M$, and ${\check \rho}_M$ is the dual of the Weil representation associated to $M$~(again, see~\cite{Sko-Weil} for a detailed explanation).

Throughout, the Fourier series coefficients of a modular form $f$, a vector-valued modular form~$f_\rho$, and a Jacobi form $\phi$ are denoted by $c(f;\, n)$, $c(f_\rho;\, n, r)$, and $c(\phi;\, n, r)$, respectively.  The theta decomposition preserves the set of Fourier series coefficients: $\{ c(\phi;\, n, r) \} = \{ c(\Theta_M(\phi); n, r ) \}$.

\subsection{Jacobi forms mod~$p$}
\label{sec:jacobiformsmodp}
Recall from the introduction that~$p \ge 5$ is a prime.  If $\phi(\tau,z) = \sum c(\phi;n,r)\, q^n \zeta^r$ and $\psi(\tau,z) = \sum c(\psi;n,r)\, q^n \zeta^r$ are Jacobi forms with coefficients in $\ZZ_{(p)}$, then $\phi\equiv \psi \pmod{p}$ when $c(\phi;n,r)\equiv c(\psi;n,r) \pmod{p}$ for all $n,r$.  Set
\begin{gather*}
  \rmJ_{k,M} (\bbF_p)
:=
  \big\{\widetilde{\phi}\, :\, \phi\in \rmJ_{k,M}(\ZZ_{(p)}) \big\}
\text{,}
\end{gather*}
where $\widetilde{\phi}(\tau,z) = \sum \widetilde{c}(\phi;\, n,r)\, q^n \zeta^r$ ($\widetilde{c}(\phi;\, n, r) \in \bbF_p$) denotes the reduction modulo $p$ of $\phi$.  If $\phi \in \rmJ_{k,M}(\ZZ_{(p)})$, then we denote its filtration modulo $p$ by
\begin{gather*}
  \omega\big(\phi\big)
:=
  \inf \big\{ k\,:\, \phi \pmod{p}\,\in \rmJ_{k,M}(\bbF_p) \big\}
\text{.}
\end{gather*}
Note that the $\bbF_p$-algebra of elliptic modular forms modulo $p$ has a natural grading with values in $\ZZ \slashdiv (p - 1) \ZZ$.  We record the following direct consequence of Theorem~2 of~\cite{SwD-l-adic}:
\begin{proposition}[\cite{SwD-l-adic}]
\label{prop:congruences-for-elliptic-modular-forms}
Let $(f_k)_k$ be a finite family of elliptic modular forms $f_k\in \rmM^{(1)}_k (\ZZ_{(p)})$.  If $\sum_k f_k \equiv 0 \pmod{p}$, then for all $a \in \ZZ \slashdiv (p - 1)\ZZ$ we have
\begin{gather*}
  \sum_{k \in a + (p - 1) \ZZ} f_k
\equiv
  0 \pmod{p}
\text{.}
\end{gather*}
\end{proposition}

Proposition~\ref{prop:congruences-for-elliptic-modular-forms} allows the following slight extension of Sofer's~\cite{Sof-JNT97} work on congruences of Jacobi forms on $\HH\times\CC$. 
\begin{proposition}
\label{prop:improved-Sofer}
Let $(\phi_k)_k$ be a finite family of Jacobi forms $\phi_k \in \rmJ_{k, m}(\ZZ_{(p)})$.  If $\sum_k \phi_k \equiv 0 \pmod{p}$, then for all $a \in \ZZ \slashdiv (p - 1) \ZZ$ we have
\begin{gather*}
  \sum_{k \in a + (p - 1) \ZZ} \phi_k
\equiv
  0 \pmod{p}
\text{.}
\end{gather*}
\end{proposition}
\begin{proof}
Recall that the ring of weak Jacobi forms of index $m$ is a module of rank $2 m$ over~$\rmM^{(1)}_\bullet$ (see $\S8$ and $\S9$ of \cite{EZ}).  In particular, if $\phi_k \in \rmJ_{k, m}(\ZZ_{(p)})$ has even weight $k$, then there are weak Jacobi forms $\phi_{-2, 1}$ and $\phi_{0, 1}$ of index~$1$ and weight $-2$ and $0$ respectively, such that 
\begin{gather*}
  \phi_k
=
  \sum_{0 \le j \le m} f_{k, j} \,
  \phi_{-2, 1}^j \phi_{0, 1}^{m - j}
\text{,} 
\end{gather*}
where $f_{k, j}\in \rmM^{(1)}_{k+2j} (\ZZ_{(p)})$.  If $k$ is odd, then $\phi_k$ can be expressed as the product of a weak Jacobi form~$\phi_{-1, 2}$ of index~$2$ and weight~$-1$ and another weak Jacobi form of even weight.  Thus, we have
\begin{gather*}
  \sum_{k \in \ZZ} \phi_k
=
  \sum_{k \in 2 \ZZ} \sum_{0 \le j \le m} f_{k, j} \,
                   \phi_{-2, 1}^j \phi_{0, 1}^{m - j}
  + 
  \phi_{-1, 2} \sum_{k \in 2 \ZZ + 1} \sum_{0 \le j \le m - 2} f_{k, j} \,
                   \phi_{-2, 1}^j \phi_{0, 1}^{m - j}
\text{.}
\end{gather*}
Exploiting the Taylor expansion with respect to~$z$ in the same way as for even weights in Lemma~2.2 of~\cite{Sof-JNT97}, we find that
\begin{gather*}
  \sum_{k \in \ZZ} f_{k, j}
\equiv
  0 \pmod{p}
\text{,}
\end{gather*}
and Proposition~\ref{prop:congruences-for-elliptic-modular-forms} yields the claim.
\end{proof}

We need the next Lemma and Proposition to prove the extension of Proposition~\ref{prop:improved-Sofer} to Jacobi forms on $\HS \times \CC^l$.

\begin{lemma}
\label{la:jacobi_restriction_vector_for_r}
Fix $0 < b \in \ZZ$ and set
\begin{gather*}
  R
:=
  \big\{ r = \rT(r_1, \ldots, r_l)\in \ZZ^l \,:\, |r_j| < b \text{ for all $1 \le j \le l$} \big\}
\text{.}
\end{gather*}
Then there exists an $s \in \ZZ^l$ such that for all $r, r' \in R$ the following holds:
\begin{gather*}
  \big( \rT s r = \rT s r' \big)
\Longleftrightarrow
  \big( r = r' \big)
\text{.}
\end{gather*}
\end{lemma}
\begin{proof}
It is easy to check that $s:= \rT(1, 4 b, \ldots, (4 b)^{l - 1})$ is as required.
\end{proof}

It will often be useful to restrict Jacobi forms as follows: If $\phi\in \rmJ_{k,M}(\ZZ_{(p)})$, $s \in \ZZ^l$, and $z' \in \CC$, then
\begin{gather*}
  \phi[s](\tau,z')
:=
  \phi(\tau, s z')
  \in \rmJ_{k, M[s]}(\ZZ_{(p)})
\text{.}
\end{gather*}

\begin{proposition}
\label{prop:jacobi_restriction_vector_for_r}
Let $(\phi_k)_{k}$ be a finite family of Jacobi forms $\phi_k\in \rmJ_{k,M}(\ZZ_{(p)})$.  If $0 \le n_0\in \ZZ$ is fixed, then there exists an $s \in \ZZ^l$ such that for all $n \le n_0$ and $r \in \ZZ^l$, we have
\begin{gather}
\label{eq:jacobi_restriction_vector_for_r:conversion-of-coeff}
  c(\phi_k[s];\, n, \rT s r)
=
  c(\phi_k;\, n, r)
\text{.}
\end{gather}

In particular, if $\phi_k \not\equiv 0 \pmod{p}$ for all~$k$, then there exists an $s \in \ZZ^l$ such that $\phi_k[s] \not\equiv 0 \pmod{p}$ for all~$k$.
\end{proposition}
\begin{proof}
Observe the condition $4 \det(M) n - M^{\#}[r] \ge 0$ in~\eqref{eq:Jacobi Fourier expansion}.  Let
\begin{gather*}
b > \max_{} \big\{ |r_j| \,:\, r = \rT(r_1, \ldots, r_l) \in \ZZ^l,\,
                            4 \det(M) n_0 - M^{\#}[r] \ge 0 \big\}
\text{,}
\end{gather*}
and apply Lemma~\ref{la:jacobi_restriction_vector_for_r} to find an $s\in\ZZ^l$ such that~\eqref{eq:jacobi_restriction_vector_for_r:conversion-of-coeff} holds.  The second statement follows from the first by choosing $n_0>\max\{n_k\}$, where $(n_k, r_k)$ are Fourier indices with minimal $n_k$ such that $c(\phi_k;\, n_k, r_k) \not\equiv0 \pmod{p}$. 
\end{proof}

We now extend Proposition~\ref{prop:improved-Sofer} to Jacobi forms on $\HS \times \CC^l$.
\begin{proposition}
\label{prop:jacobi_weight_congruence}
Let $\phi \in \rmJ_{k,M}(\ZZ_{(p)})$ and $\psi \in \rmJ_{k',M'}(\ZZ_{(p)})$ such that $0 \not \equiv \phi \equiv \psi \pmod{p}$.  Then $M = M'$ and $k \equiv k' \pmod{(p - 1)}$.

Moreover, if $M$ is fixed and $(\phi_k)_{k}$ is a finite family of Jacobi forms $\phi_k\in \rmJ_{k,M}(\ZZ_{(p)})$ such that $\sum_k \phi_k \equiv 0 \pmod{p}$, then for all $a \in \ZZ \slashdiv (p - 1) \ZZ$ we have
\begin{gather*}
  \sum_{k \in a + (p - 1)\ZZ} \phi_k
\equiv
  0 \pmod{p}
\text{.}
\end{gather*}
\end{proposition}

\begin{proof}
Let $\phi \in \rmJ_{k,M}(\ZZ_{(p)})$ and $\psi \in \rmJ_{k',M'}(\ZZ_{(p)})$ such that $0 \not \equiv \phi \equiv \psi \pmod{p}$.  We prove the equality of indices as in the case of $l=1$ (see Lemma 2.1 of \cite{Sof-JNT97}).  More precisely, for any $\lambda \in \ZZ^l$ we have
\begin{gather}
\label{eq:jacobi-weight-congruence}
  q^{-M[\lambda]} \zeta^{-2 M \lambda} \phi
\equiv
  q^{-M'[\lambda]} \zeta^{-2 M' \lambda} \psi
  \equiv
  q^{-M'[\lambda]} \zeta^{-2 M' \lambda} \phi \pmod{p}
\text{.}
\end{gather}
We find that $M[\lambda] = M'[\lambda]$ for all $\lambda \in \ZZ^l$, and hence $M = M'$. The congruence $k \equiv k' \pmod{(p - 1)}$ follows from the second part.

Let $M$ be fixed and $\phi_k\in \rmJ_{k, M}(\ZZ_{(p)})$ such that $\sum_k \phi_k \equiv 0 \pmod{p}$.  Note that if $s \in \ZZ^l$, then $\phi_k[s]\in\rmJ_{k, M[s]}(\ZZ_{(p)})$.  Consider the sum
\begin{gather*}
  \sum_k \phi_k[s]
\equiv
  0 \pmod{p}
\text{,}
\end{gather*}
and apply Proposition~\ref{prop:improved-Sofer} to find that
\begin{gather}
\label{sum phi_k}
  \sum_{k \in a + (p - 1) \ZZ} \hspace{-.5em}
  \phi_k[s]
\equiv
  0 \pmod{p}
\text{.}
\end{gather}
If $0 \le n_0\in \ZZ$ is fixed, then Proposition~\ref{prop:jacobi_restriction_vector_for_r} asserts that there exists an $s \in \ZZ^l$ such that for all $n \le n_0$ and $r \in \ZZ^l$, we have $c(\phi_k[s]; n, \rT s r)=c(\phi_k; n, r)$.  In particular, (\ref{sum phi_k}) holds also for that $s$.   Thus, for arbitrary $n$ and $r$ we have
\begin{gather*}
  \sum_{k \in a + (p - 1) \ZZ} \hspace{-.5em}
  c(\phi_k;\, n, r)
\equiv
  0 \pmod{p}
\end{gather*}
and hence
\begin{gather*}
  \sum_{k \in a + (p - 1) \ZZ} \hspace{-.5em}
  \phi_k
\equiv
  0 \pmod{p}
\text{.}
\qedhere
\end{gather*}
\end{proof}

\begin{remark}
One can repeatedly employ~\eqref{eq:jacobi-weight-congruence} to find that Proposition~\ref{prop:jacobi_weight_congruence} can be extended as follows:  If $(\phi_{k, M})_{k,M}$ is a finite family of Jacobi forms $\phi_{k,M}\in \rmJ_{k,M}(\ZZ_{(p)})$ 
with $\sum_{k, M} \phi_{k, M} \equiv 0 \pmod{p}$, then for every fixed~$M$ and $a \in \ZZ \slashdiv (p - 1) \ZZ$ we have
\begin{gather*}
  \sum_{k \in a + (p - 1)\ZZ} \phi_{k, M}
\equiv
  0 \pmod{p}
\text{.}
\end{gather*}
\end{remark}

A priori, $\rmJ_{k, M}$ has a basis of Jacobi forms with Fourier series expansions over~$\CC$.  Our final result in this section shows that there exists a basis with integral Fourier series coefficients, which is important for arithmetic applications.
\begin{theorem}
\label{thm:generators-of-rings-of-jacobi-forms}
There exists a basis of $\rmJ_{k, M}$ with integral Fourier series coefficients.  That is, for all $k$ and $M$, we have $\rmJ_{k, M} = \rmJ_{k, M}(\ZZ) \otimes \CC$.
\end{theorem}
\begin{proof}
If $l=1$, then it is well known (see \cite{EZ}) that $\rmJ_{k, m}$ has a basis with Fourier series coefficients in~$\ZZ$.  For ~$l > 0$ it suffices to show that $\rmJ_{k, M} = \rmJ_{k, M}(\QQ) \otimes \CC$, since $\rmJ_{k, M}(\ZZ)$ is torsion free.  Let $\rmF\rmJ_{k, M}$ be the~$\CC$ vector space of formal Fourier series expansions
\begin{gather*}
  \sum_{\substack{ n \ge 0,\, r \in \ZZ^l \\ 4 \det(M) n - M^{\#}[r] \ge 0 }} \hspace{-1.5em}
  c( \phi;\, 4 \det(M) n - M^{\#}[r], \ov{r})\, q^n \zeta^r
\text{,}
\end{gather*}
where $c( \phi;\, D, \ov{r})$ depends on the reduction~$\ov{r} \pmod{2 M}$ of~$r$.  Note that $\rmJ_{k, M} \subset \rmF\rmJ_{k, M}$ by the map which sends a Jacobi form to its Fourier series expansion.

There are formal restriction maps $(\,\cdot\,)[s]$, that map a formal Fourier series expansion to
\begin{gather*}
  \phi[s]
=
  \sum_{\substack{ n \ge 0,\, r \in \ZZ^l \\ 4 \det(M) n - M^{\#}[r] \ge 0 }} \hspace{-1.5em}
  c( \phi;\, 4 \det(M) n - M^{\#}[r], \ov{r})\, q^n \zeta^{\prime\, \rT s r}
\text{,}
\end{gather*}
where $\zeta' := \exp(2 \pi i\, z')$ with $z' \in \CC$.  Write $(\,\cdot\,)[s]^{-1}$ for the preimage under restriction along~$s$. Lemma~4.1 of~\cite{Ra12-special-cycles} and Proposition~4.8 of~\cite{Ra12-special-cycles} imply that there is a finite set $\cS$ such that
\begin{gather*}
  \bigcap_{s \in \cS} (\,\cdot\,)[s]^{-1}\big( \rmJ_{k, M[s]} \big)
=
  \rmJ_{k, M}
\subset
  \rmF\rmJ_{k, M}
\text{.}
\end{gather*}  
The preimage $(\,\cdot\,)[s]^{-1}$ preserves rationality of Fourier series coefficients, because $(\,\cdot\,)[s]$ is defined over~$\ZZ$.  In other words, we have
\begin{gather*}
  \bigcap_{s \in \cS} (\,\cdot\,)[s]^{-1}\big( \rmJ_{k, M[s]}(\QQ) \big)
=
  \rmJ_{k, M}(\QQ)
\text{.}
\qedhere
\end{gather*}
\end{proof}

\subsection{The module $\rmJ_{\bullet, M}(\ZZ_{(p)})$}
\label{sec:expceptional-set-for-jacobi-forms}
In this section, we show that (for fixed $M$) the $\rmM^{(1)}_\bullet(\ZZ_{(p)})$-module
\begin{gather*}
  \rmJ_{\bullet, M}(\ZZ_{(p)})
=
  \bigoplus_{k \in \ZZ}
  \rmJ_{k, M}(\ZZ_{(p)})
\text{}
\end{gather*}
is free for all~$p\geq 5$.
\begin{definition}
Let $\cP(M)$ be the set of primes~$p$ such that $\rmJ_{\bullet, M}(\ZZ_{(p)})$ is a free module over $\rmM^{(1)}_{\bullet}(\ZZ_{(p)})$.
\end{definition}

We first characterize the set $\cP(M)$ by saturation properties.  Over characteristic~$0$, we say that a submodule $N \subseteq N'$ is $p$-saturated (in $N'$), if for every $\phi \in N'$ with $p \phi\in N$, we have $\phi \in N$.
\begin{proposition}
\label{prop:characterization-of-exceptional-congruence-primes}
We have
\begin{multline*}
  \cP(M)
=
  \Big\{
    p\;\text{prime} \,:\,
    \forall\,k \in \ZZ \,:\,
    \sum_{\substack{k^\rmM + k^\rmJ = k \\[0.1em] k^\rmJ < k}}\!\!
    \rmM^{(1)}_{k^\rmM}(\ZZ_{(p)})\, \rmJ_{k^\rmJ, M}(\ZZ_{(p)})\;
    \text{is $p$-saturated in $\rmJ_{k, M}(\ZZ_{(p)}$})
  \Big\}
\text{.}
\end{multline*}
\end{proposition}
\begin{proof}
Write $\cP_{sat}$ to denote the set on the right hand side of the claim.  Note that $\cP(M) \subset \cP_{sat}$.  

Let $p\in \cP_{sat}$.  We will show that $\rmJ_{\bullet, M}(\ZZ_{(p)})$ is free, i.e., $p \in \cP(M)$.  We recursively define generators.  Since~$M$ is non-degenerate, we have $\rmJ_{0, M} = \{0\}$.  Fix~$0 < k \in \ZZ$, and suppose that there are algebraically independent (over $\rmM^{(1)}_{\bullet}$) Jacobi forms~$\phi_i \in \rmJ_{k_i, M}(\ZZ_{(p)})$ ($1\leq i\leq d$) with $k_i < k$ that span $\rmJ_{k', M}$ for all $k' < k$.  Then the submodule
\begin{gather*}
 N_p:= \bigoplus_{i = 1}^d \rmM^{(1)}_{k - k_i}(\ZZ_{(p)})\, \phi_i
=
  \sum_{\substack{k^\rmM + k^\rmJ = k \\[0.1em] k^\rmJ < k}}\!\!
  \rmM^{(1)}_{k^\rmM}(\ZZ_{(p)})\, \rmJ_{k^\rmJ, M}(\ZZ_{(p)})
\subseteq
  \rmJ_{k, M}(\ZZ_{(p)})
\end{gather*}
is $p$\nbd saturated, since $p\in\cP_{sat}$.  Hence there exists a $\ZZ_{(p)}$\nbd complement~$W_p$ of~$N_p$ in~$\rmJ_{k, M}(\ZZ_{(p)})$.  Fix a basis of $W_p$, say $\phi_{d + 1}, \ldots, \phi_{d + d'}$.  Note that the elements~$\phi_{1}, \ldots, \phi_{d + d'}$ of~$\rmJ_{\bullet, M}(\ZZ_{(p)})$ generate a free $\rmM^{(1)}_{\bullet}(\ZZ_{(p)})$\nbd submodule of $\rmJ_{\bullet, M}(\ZZ_{(p)})$, since $\rmJ_{\bullet, M}$ is free over~$\rmM^{(1)}_{\bullet}$ (see~\cite{M-M-London10}).  We conclude that $\rmJ_{\bullet, M}(\ZZ_{(p)})$ is free.
\end{proof}

Theorem~1.1 of Mason and Marks~\cite{M-M-London10} implies that $\rmJ_{\bullet,M}$ is a free module of rank~$\rk_M := \det(2 M)$ over $\rmM^{(1)}_{\bullet}$.  In particular, if $\phi \in \rmJ_{k,M}$, then Theorem~\ref{thm:generators-of-rings-of-jacobi-forms} asserts that there exists $\phi_i \in \rmJ_{k_i,M}(\ZZ)$ and $f_i\in \rmM^{(1)}_{k-k_i}$ ($1\leq i\leq \rk_K$) such that $\phi = \sum_{i=1}^{\rk_M} f_i \phi_i$. The next Proposition addresses the elliptic modular forms $f_i$ in the case that $\rmJ_{\bullet, M}(\ZZ_{(p)})$ is free.

\begin{proposition}
\label{prop:jacobi-p-integral-coefficients}
Let $\phi \in \rmJ_{k,M}(\ZZ_{(p)})$ with $\phi = \sum_{i=1}^{\rk_M} f_i \phi_i$.  If $p \in \cP(M)$, then the elliptic modular forms $f_i$ have also $p$-integral rational coefficients.  Moreover, if $p \in \cP(M)$ and $\psi = \sum_{i=1}^{\rk_M} g_i \phi_i \in \rmJ_{k',M}(\ZZ_{(p)})$ such that $0 \not\equiv \phi \equiv \psi \pmod{p}$, then $f_i \equiv g_i \pmod{p}$.
\end{proposition}
\begin{proof}
It suffices to show that if $\phi = \sum_{i=1}^{\rk_M} f_i \phi_i\equiv 0 \pmod{p}$, then $f_i \equiv 0 \pmod{p}$ for all~$i$.  Note that $\phi_i \in \rmJ_{k_i,M}(\ZZ)$.  Proposition~\ref{prop:characterization-of-exceptional-congruence-primes} asserts that the module $\sum_{i=1}^{\rk_M} \rmM^{(1)}_{k - k_i}(\ZZ_{(p)})\, \phi_i$ is $p$\nbd saturated.  Hence $\frac{1}{p} \phi = \sum_{i=1}^{\rk_M} \phi_i g_i$ for some $g_i \in \rmM^{(1)}_{k - k_i}(\ZZ_{(p)})$.  The Jacobi forms $\phi_i$ form a basis of $\rmJ_{\bullet, M}(\ZZ_{(p)})$, and we conclude that $\frac{1}{p} f_i = g_i \in \rmM_{k - k_i}(\ZZ_{(p)})$, as desired.
\end{proof}

To determine which primes belong to~$\cP(M)$, we reconsider parts of~\cite{Ra12-special-cycles} in a different spirit.  The space of formal Fourier series expansions of even or odd weight and index~$M$ is given by
\begin{multline*}
  \cF\cE(M)^{{\rm even}/{\rm odd}}
\\
:=
  \Big\{ \sum_{\substack{n \in \ZZ,\, r \in \ZZ^l \\ 4 \det(M) n - M^{\#}[r] \ge 0}}\hspace{-1.5em}
  c(n, r)\, q^n \zeta^r
     \,:\, c(n, r) = c(n + \rT \lambda r + M[\lambda], r + 2 M \lambda ),\,
           c(n, r) = \pm c(n, -r)
  \Big\}
\text{,}
\end{multline*}
where $c(n,r)=c(n, -r)$ in the even case and $c(n,r)=-c(n, -r)$ in the odd case.

For a finite set~$\cS$ of vectors in~$\ZZ^l$, define
\begin{gather}
\label{eq:def:jacobi-restriction-module}
  \cF\cE(M, \cS)^{\rm even}
:=
  \Big( \bigoplus_{s \in \cS} (\,\cdot\,)[s] \Big)
  \big(\cF\cE(M)^{\rm even}\big)
\subset
  \bigoplus_{s \in \cS} \cF\cE(m)^{\rm even}
\text{.}
\end{gather}
The space~$\cF\cE(M, \cS)^{\rm odd}$ is defined analogously.

Let $\cR$ be a set of representatives of~$\ZZ^l \slashdiv {2 M \ZZ^l}$ with the following property:  For every $r \in \cR$ and every $r' \equiv r \pmod{M \ZZ^l}$, we have $M^{-1}[r] \le M^{-1}[r']$.  In~\cite{Ra12-special-cycles} it is shown that the elements in $\cF\cE(M)^{{\rm even}/{\rm odd}}$ are uniquely determined by their coefficients $c(n, r)$ with $r \in \cR$.

\begin{lemma}
\label{la:saturation-and-vanishing}
Fix a prime~$p \ge 5$.  There exists a set~$\cS$ of integral $l$-vectors such that
\begin{enumerate}[(i)]
\item The map $(\,\cdot\,)[\cS]$ is injective.
\item Let $\phi:\HS \times \CC^l\ra\CC$ with Fourier series expansion in~$\cF\cE(M)^{{\rm even}/{\rm odd}}$. Then $\phi = 0$ if and only if~$\phi[s] = 0$ for all~$s \in \cS$.
\item $\cF\cE(M, \cS)^{\rm even}$ and $\cF\cE(M, \cS)^{\rm odd}$ are $p$\nbd saturated.
\end{enumerate}
\end{lemma}
\begin{proof}
The second part is given by Corollary~4.7 of~\cite{Ra12-special-cycles}.  To show the first and third parts, we write
\begin{gather*}
  \ov{\cR}
:=
  \big\{ r' \in \ZZ^l \,:\,
     \exists r \in \cR \,:\,
     r \equiv \pm r' \pmod{2 M \ZZ^l},\, M^{\#}[r] = M^{\#}[r']
  \big\}
\text{.}
\end{gather*}
Choose $b \in \ZZ$ such that $|r_i| < b$ for all $r=\rT(r_1, \ldots, r_l)\ \in \ov{\cR}$.  Lemma~\ref{la:jacobi_restriction_vector_for_r} provides a set~$\cS$ such that for all~$r \in \ov{\cR}$ there exists an~$s \in \cS$ such that $\rT s r = \rT s r'$ for some $r' \in \cR$ implies~$r = r'$.  Thus, the first part holds.

We prove the third part only for even weights, and the case of odd weights is established in exactly the same way.  Suppose that there exists a $\phi \in \cF\cE(M, \cS)^{\rm even}$ such that $\tfrac{1}{p}\phi \not\in \cF\cE(M, \cS)^{\rm even}$, but $\tfrac{1}{p}\phi \in \bigoplus_s \cF\cM(M[s])^{\rm even}$.  Fix a preimage~$\psi \in \cF\cM(M)^{\rm even}$ of $\phi$ under~$\bigoplus_s (\,\cdot\,)[s]$.  By assumption, $\tfrac{1}{p}\psi \not\in \cF\cM(M)^{\rm even}$.  As usual, denote the Fourier series coefficients of $\psi$ by $c(\psi;\, n, r)$.   Let~$n_0$ be minimal subject to the condition that there exists an $r_0$ with $c(\psi;\, n_0, r_0) \not\equiv 0 \pmod{p}$.  The definitions of $\cR$ and~$\cF\cM(M)^{\rm even}$ imply that $c(\psi;\, n_0, r) \equiv 0 \pmod{p}$ for all $r \not\in \ov{\cR}$, and hence $r_0 \in \ov{\cR}$.  Moreover, there exists an $s \in \cS$ such that $c(\psi[s]; n_0, \rT s r_0) \equiv c(\psi; n_0, r_0) \pmod{p}$, which gives the contradiction that $\tfrac{1}{p}\phi \not\in \bigoplus_s \cF\cE(m)^{\rm even}$.  This proves the third part for even weights.
\end{proof}
\begin{corollary}
\label{cor:jacobi-restriction-isomorphism}
Consider the map
\begin{gather*}
  (\,\cdot\,)[\cS]
=
  \bigoplus_{s \in \cS} (\,\cdot\,)[s] :\,
  \cF\cE(M)^{\rm even}
\longrightarrow
  \bigoplus_{s \in \cS} \cF\cE(M[s])^{\rm even}
\text{.}
\end{gather*}
Identify $\ZZ_{(p)}$-modules of Jacobi forms with the associated modules of Fourier series expansions.  For any~$\cS$ as in Lemma~\ref{la:saturation-and-vanishing}, we have an $\rmM^{(1)}_{\bullet}(\ZZ_{(p)})$-isomorphism
\begin{gather}
\label{eq:module-isomorphism}
  \rmJ_{2\bullet, M}(\ZZ_{(p)})
=
  (\,\cdot\,)[\cS]^{-1}
  \Big(
  \bigoplus_{s \in \cS} \rmJ_{2 \bullet, M[s]}(\ZZ_{(p)})
  \Big)
\text{.}
\end{gather}

An analogous statement holds for $\rmJ_{2 \bullet + 1, M}(\ZZ_{(p)})$.
\end{corollary}
\begin{proof}
As before, we only consider the case of even weights.  The first and second property in Lemma~\ref{la:saturation-and-vanishing} can be used (see Section~4 of~\cite{Ra12-special-cycles} for details) to show that
\begin{gather*}
  \rmJ_{2 \bullet, M}
=
  (\,\cdot\,)[\cS]^{-1}
  \Big(
  \bigoplus_{s \in \cS} \rmJ_{2 \bullet, M[s]}
  \Big)
\text{.}
\end{gather*}

We have to study $p$\nbd integrality under $(\,\cdot\,)[\cS]$ in order to show the equality in \eqref{eq:module-isomorphism}.  Note that restrictions of formal Fourier series expansions with $p$\nbd integral coefficients have $p$\nbd integral coefficients.  Hence the left hand side of \eqref{eq:module-isomorphism} is contained in the right hand side.  To see the opposite inclusion, suppose there was a $\phi \in \rmJ_{2 \bullet, M}(\ZZ_{(p)})$ whose restriction $\phi[\cS]$ is divisible by~$p$.  Then $\phi$ itself is divisible by~$p$, because $\cF\cE(M, \cS)^{\rm even}$ is $p$\nbd saturated by the third property in Lemma~\ref{la:saturation-and-vanishing}.  This proves the equality.

The following equation shows that it gives an isomorphism of~$\rmM^{(1)}_\bullet(\ZZ_{(p)})$\nbd modules.  Given formal Fourier series expansions $\sum_{n, r} c(\phi;\, n, r)\, q^n \zeta^r$ and $\sum_n c(f;\, n)\, q^n$, we have
\begin{gather*}
  (f \phi)[s]
=
  \sum_{n', n, r} c(f;\, n')\, c(\phi;\, n, r)\, q^{n + n'} \zeta^{\rT s r}
=
  f\, \phi[s]
\text{.}
\qedhere
\end{gather*}
\end{proof}

\begin{theorem}
\label{thm:exception-set-is-trivial}
For every~$M$, we have~$\cP(M) \subseteq \{2, 3\}$.
\end{theorem}
\begin{proof}
Recall that $p \ge 5$.  The case~$l = 1$ is classical (see~\cite{EZ} and~\cite{Sof-JNT97}).  Assume that~$l > 1$.  Choose~$\cS$ as in Lemma~\ref{la:saturation-and-vanishing}.  We apply the isomorphism given in Corollary~\ref{cor:jacobi-restriction-isomorphism}.  We have to prove that
\begin{gather*}
  \sum_{\substack{k^\rmM + k^\rmJ = k \\ k^\rmJ < k}}\!\!
  \rmM_{k^\rmM}(\ZZ_{(p)})\, \rmJ_{k^\rmJ, M}(\ZZ_{(p)})
\subseteq
  \rmJ_{k, M}(\ZZ_{(p)})
\end{gather*}
is a $p$\nbd saturated submodule for every~$k$.  Using the above isomorphism, this amounts to showing that
\begin{gather*}
  \big( \cF\cE(M, \cS)^{\rm even} \cup \cF\cE(M, \cS)^{\rm odd} \big)
  \;\cap\;
  \Big(
  \bigoplus_{s \in \cS}
  \sum_{\substack{k^\rmM + k^\rmJ = k \\ k^\rmJ < k}}\!\!
  \rmM_{k^\rmM}(\ZZ_{(p)})\, \rmJ_{k^\rmJ, M[s]}(\ZZ_{(p)})
  \Big)
\end{gather*}
is $p$\nbd saturated.  Lemma~\ref{la:saturation-and-vanishing} implies that the first module in the intersection is $p$\nbd saturated.  Moreover, each term in the direct sum of the second module is $p$\nbd saturated due to the case~$l = 1$ of this theorem.  The intersection of two $p$\nbd saturated modules is $p$\nbd saturated, and the claim follows for all~$l$.
\end{proof}

\subsection{Heat cycles and $U_p$ congruences}
\label{sec:HeatcyclesandUcongruences}
In this section, we investigate heat cycles of Jacobi forms, and we determine conditions for $U_p$ congruences of Jacobi forms.   Consider the heat operator
\begin{gather*}
  \bbL
:=
  \bbL_M
:=
  \frac{1}{(2\pi i)^2}
  \left(
  8 \pi i \det(M) \partial_\tau - M^{\#} [ \partial_z ]
  \right)
\text{.}
\end{gather*}
Note that if $l=1$ and $M =m$, then $\bbL_M=\frac{1}{(2 \pi i)^2}\left(8 \pi im\frac{\partial}{\partial\tau}-\frac{\partial^2}{\partial z^2}\right)$ is the usual heat operator.  If $\phi \in \rmJ_{k,M}$, then a direct computation (see also Lemma 3.3 of \cite{Ch-Kim}) shows that
\begin{gather}
\label{heat-trans}
  \bbL_M ( \phi )
=
  \frac{(2k - l)\,\det(M)}{6}\,\phi E_2 + \widehat{\phi}
\text{,}
\end{gather}
where $E_2$ is the quasimodular Eisenstein series of weight $2$ and where $\widehat{\phi} \in \rmJ_{k+2,M}$.

Tate's theory of theta cycles (see $\S7$ of \cite{Joch-1982}) relies on  Lemma~5 of~\cite{SwD-l-adic}, which gives the filtration of the theta operator applied to a modular form.  Proposition 2 of \cite{heat} extends Lemma~5 of~\cite{SwD-l-adic} to Jacobi forms on $\HS \times \CC$, and our next proposition extends this further to Jacobi forms on $\HS \times \CC^l$. 

\begin{proposition}
\label{prop: filtration-of-heat-application}
If $\phi \in \rmJ_{k,M}(\ZZ_{(p)})$, then $\bbL(\phi) \pmod{p}$ is the reduction of a Jacobi form modulo~$p$.  Moreover, we have
\begin{gather}
  \omega\big(\bbL_m(\phi)\big)
\leq
  \omega(\phi) + p + 1
\text{,}
\end{gather}
with equality if and only if $p \nisdiv (2\omega(\phi) - l ) \det(2 M)$.
\end{proposition}
\begin{proof}
We proceed exactly as in \cite{heat}, and we assume that $\omega(\phi)=k$. Let
\begin{gather*}
  E_{k}(\tau)
:=
  1
  -
  \tfrac{2 k}{B_{k}}
  \sum_{n=1}^{\infty} \Big( \sum_{0<d\isdiv n}d^{k-1} \Big)\, q^n
\end{gather*}
denote the usual Eisenstein series.  Recall that $E_{p-1} \equiv 1 \pmod{p}$ and $E_2 \equiv E_{p+1} \pmod{p}$. Equation~\eqref{heat-trans} shows that $\bbL_m(\phi) \in\rmJ_{k+p+1,M} (\bbF_p)$, i.e., 
$\omega\big(\bbL_m(\phi)\big)\leq k+p+1$. 

If $p$ divides $\left(2k-l\right)\det(2M)$, then $\omega\big(\bbL_m(\phi)\big)\leq k+2<k+p+1$ by~\eqref{heat-trans}.  On the other hand, 
if $\omega\big(\bbL_m(\phi)\big) < k+p+1$, then $\omega\left(\frac{(2k-l)\,\det(M)}{6}\,\phi E_2\right)< k+p+1$ by~\eqref{heat-trans}.  
It remains to show that $\omega\left(\phi E_2\right)=k+p+1$, which then implies that $p$ divides $\left(2k-l\right)\det(2M)$.  Proposition~\ref{prop:jacobi-p-integral-coefficients} asserts that $\phi = \sum_{i=1}^{\rk_M} f_i \phi_i$, where the elliptic modular forms $f_i$ have $p$-integral rational coefficients. There exists an $f_j$ such that $\omega\big(f_j\phi_i\big)=k$, since otherwise $\omega(\phi)<k$. Theorem~2 and Lemma~5 of~\cite{SwD-l-adic} guarantee that $f_iE_2$ has maximal filtration, and we find that $\omega\left(\phi E_2\right)=k+p+1$, which completes the proof.
\end{proof}

Let us introduce an analog of Atkin's $U$-operator for Jacobi forms of higher degree:
\begin{definition}
For
\begin{align*}
  \phi(\tau,z)
=
  \sum_{\substack{ n \geq 0, r \in \ZZ^l \\
                   4 \det(M)\, n - M^{\#}[r] \geq 0 }}
  \hspace{-1em}
  c(n,r)\, q^n\zeta^r \in \rmJ_{k,M}
\text{,}
\end{align*}
we define
\begin{gather}
  \phi(\tau,z) \big|\, U_p
:=
  \sum_{\substack{ n\geq 0, r\in \ZZ^l \\
                   4 \det(M)\, n - M^{\#}[r] \geq 0\\
                   p \isdiv (4 \det(M)\, n - M^{\#}[r])}}
  \hspace{-1em}
  c(n,r)\, q^n\zeta^r
\text{.}
\end{gather}
\end{definition}

If $l=1$, then this is precisely the $U_p$ operator of \cite{crit,heat}.  The following Theorem on $U_p$ congruences is the main result in this Section.  Note that the results for $l$ even and $l$ odd are quite different. 

\begin{theorem}
\label{thm:Jacobi-U_p}
Assume that $p\geq k$ such that $p \nisdiv \det(2  M)$, and let $\phi \in \rmJ_{k,M}(\ZZ_{(p)})$ such that $\phi \not\equiv 0\pmod{p}$.
\begin{enumerate}[a)]
\item 
Let $l$ be even.  If $p>\frac{l}{2}+1$, then
\begin{gather*}
  \omega\Big(\bbL^{p+1-k+\frac{l}{2}}(\phi)\Big)
=
  \begin{cases}
    2p+2+l-k\text{,} & \text{if $\phi \big|\, U_p \not\equiv 0 \pmod{p}$;}
    \\
    \hphantom{2}p+3+l-k\text{,} & \text{if $\phi \big|\, U_p \equiv 0 \pmod{p}$.}
  \end{cases}
\end{gather*}

\item
Let $l$ be odd, and $k > \frac{l + 5}{2}$.
\begin{enumerate}[i)]
\item If $p > 2k-l-4$, then $\phi \, \big | \, U_p \not\equiv 0 \pmod{p}$.
\item If $p < 2k - l - 4$, then
\begin{gather*}
  \omega\Big(\bbL^{\frac{3p+l}{2}+1-k}(\phi)\Big)
=
  \begin{cases}
    3p+2+l-k\text{,} & \text{if $\phi \big|\, U_p \not\equiv 0 \pmod{p}$;}
    \\
    2p+3+l-k\text{,} & \text{if $\phi \big|\, U_p \equiv 0 \pmod{p}$.}
  \end{cases}
\end{gather*}
\end{enumerate}
\end{enumerate}
\end{theorem}

\begin{proof}
We follow the proof of Proposition~3 in \cite{crit}, and our key ingredients are Propositions~\ref{prop:jacobi_weight_congruence} and~\ref{prop: filtration-of-heat-application}.  

Suppose that $\phi \big|\, U_p \equiv 0 \pmod{p}$, in which case $\bbL^{p-1}(\phi)\equiv \phi \pmod{p}$, i.e., $\phi$ is in its own heat cycle.  We say that $\phi_1$ a {\it low point} of its heat cycle if $\phi_1=\bbL^A(\phi)$ and $2\omega\big(\bbL^{A-1}\phi\big)\equiv l \pmod{p}$. Let $\phi_1$ be a low point of its heat cycle and let $c_j\in\NN$ be minimal such that 
\begin{gather*}
  2\omega\big(\bbL^{c_j-1}(\phi_1)\big)
=
  2\left(\omega\big(\phi_1\big)+(c_j-1)(p+1)\right)\equiv l \pmod{p}
\text{,}
\end{gather*}
and let $b_j\in\NN$ be defined by
\begin{gather*}
  \omega\big(\bbL^{c_j}(\phi_1)\big)
=
  \omega\big(\phi_1\big)+c_j(p+1)-b_j(p-1)
\text{.}
\end{gather*}
Exactly as in $\S7$ of \cite{Joch-1982} (see also $\S3$ of \cite{crit}), one finds that there is either one fall with $c_1=p-1$ and $b_1=p+1$, or there are two falls with $b_1=p-c_2$ and $b_2=p-c_1$. One fall occurs if and only if $2\omega\big(\phi_1\big)\equiv l+4 \pmod{p}$.  In the following, we assume that there are two falls, and we write $\omega\big(\phi_1\big)=ap+B$ with $1\leq B\leq p$ and $p\neq 2B-4-l$.  Especially, if $\phi_1=\phi$, then $a=0$ and $B=k$.  We find that 
\begin{gather*}
2c_1+2B-2-l\equiv 0\pmod{p}.
\end{gather*}

If {\it $l$ is even}, then necessarily $c_1=1-B+\frac{l}{2}$ or $c_1=p+1-B+\frac{l}{2}$.  Note that $c_1\geq 1$, and hence the case $c_1=1-B+\frac{l}{2}$ is only possible if $B\leq\frac{l}{2}$.  In particular, if $\phi_1=\phi$, then $k=B=\frac{l}{2}$, since $\rmJ_{k,M}=\{0\}$ if $k<\frac{l}{2}$.  However, if $k=\frac{l}{2}$, then
\begin{gather*}
\omega\big(\bbL^{1}(\phi)\big)=\tfrac{l}{2}+(p+1)-2(p-1)=3+\tfrac{l}{2}-p\underset{p>\frac{l}{2}+1}{<}2
\text{,}
\end{gather*}
which is impossible, since $\rmJ_{1,M}=\{0\}$.
For the case $c_1=p+1-B+\frac{l}{2}$ we obtain
\begin{gather*}
\omega\big(\bbL^{c_1}(\phi_1)\big)=(a+1)p+3+l-B
\text{,}
\end{gather*} 
which gives the desired formula if $\phi_1=\phi$.

If {\it $l$ is odd}, then necessarily $c_1=\frac{p+l}{2}+1-B$ or $c_1=\frac{3p+l}{2}+1-B$.  If $c_1=\frac{p+l}{2}+1-B$, then $b_1=\frac{3p+l}{2}+2-B$, and if, in addition, $\phi_1=\phi$, then
\begin{gather*}
\omega\big(\bbL^{c_1}(\phi)\big) =k+c_1(p+1)-b_1(p-1)=3+l-k\underset{k>\frac{l+5}{2}}{\leq}\tfrac{l-1}{2}
\text{,}
\end{gather*}
which is impossible, since $\rmJ_{k,M}=\{0\}$ if $k<\frac{l}{2}$.  For the case $c_1=\frac{3p+l}{2}+1-B$ we calculate
\begin{gather*}
\omega\big(L^{c_1}(\phi_1)\big)=(a+2)p+3+l-B
\text{,}
\end{gather*} 
which is as required if $\phi_1=\phi$.  Note also that in this case $c_1\geq 1$ implies that $p<2B-l-4$.  Hence if $p > 2k-l-4$, then $\phi \, \big | \, U_p \not\equiv 0 \pmod{p}$.

On the other hand, if $\phi \big|\, U_p \not\equiv 0 \pmod{p}$, then (under the assumption that $k\leq p$) it is easy to see that $\bbL(\phi)$ is a low point of its heat cycle (see also \cite{crit}) with $\omega\big(\bbL(\phi)\big)=p+k+1$.  Thus, we can apply our previous findings with $a=1$ and $B=k+1$:  If {\it $l$ is even}, then the case $c_1=1-B+\frac{l}{2}$ is impossible, since $c_1\geq 1$ would imply $k<\frac{l}{2}$.  Hence $c_1=p+1-B+\frac{l}{2}$, and $\bbL^{p+1-k+\frac{l}{2}}(\phi)=2p+2+l-k$.  If {\it $l$ is odd}, then necessarily $p<2B-l-4$, and the case $c_1=\frac{p+l}{2}+1-B$ is impossible (since again $c_1\geq 1$).  Consequently, $c_1=\frac{3p+l}{2}+1-B$ and we obtain $\bbL^{\frac{3p+l}{2}+1-k}(\phi)=3p+2+l-k$.

\end{proof}

\begin{remark}
\label{rem:U_p-explanation}
If $l$ is odd, then Theorem~\ref{thm:Jacobi-U_p} generalizes Proposition 3 of \cite{crit} and Theorem~2 of~\cite{C-C-R-Siegel}.  If $l$ is even, then the result is weaker, since one cannot exclude $U_p$ congruences modulo $p$ for almost all primes $p$.  This is related to the fact that $\rmM^{(1)}_k \cong \rmJ_{k + \frac{l}{2}, M}$ for unimodular $M$ (in which case $8 \isdiv l$).  Indeed, for even~$l$ the situation is somewhat analogous to the case of elliptic modular forms (see $\S7$ of \cite{Joch-1982}), where it is also impossible to exclude $U_p$ congruences modulo $p$ for almost all primes $p$.
\end{remark}

\subsection{Examples}
\label{sec:Jacobiexamples}
We discuss a few examples to illustrate Theorem~\ref{thm:Jacobi-U_p} in the case that $l=3$ and
\begin{gather*}
  M
=
  \left(\begin{smallmatrix} 1 & 0 & 0 \\
                            0 & 1 & 0 \\
                            0 & 0 & 1
  \end{smallmatrix}\right)
\text{,}
\end{gather*} 
where we have performed necessary calculations with the help of Sage~\cite{sage57}.  Note that \cite{Ra12-special-cycles} provides tools to compute spaces of Jacobi forms:  The dimension formulas of $\rmJ_{k, M}$ are implemented in the computer code accompanying~\cite{Ra12-special-cycles}.  Moreover, Lemma~4.1 of~\cite{Ra12-special-cycles} shows that any Jacobi form in $\rmJ_{k, M}$ is uniquely determined by its restrictions along vectors in a finite set $\cS \subset \ZZ^l$.  For $l=3$ and $M$ as above, one finds using Sage that
\begin{gather*}
  \cS
=
  \big\{ (0, 1, 1), (1, 1, 0), (1, 0, 0), (0, 0, 1),
         (1, 0, 1), (0, 1, 0), (1, 1, 1)
  \big\}
\end{gather*}
is a possible choice.  The index of a restriction along $s\in\ZZ^3$ equals $M[s]$.  In our case, $M[s]=3$ is the maximal index that can occur for $s \in \cS$.  Note that Jacobi forms of index~$m$ are determined by corrected Taylor coefficients (see Section~3 of~\cite{EZ}) in $\rmM^{(1)}_{k + 2j}$ for $0 \le j \le m$.  Hence we can deduce a Sturm bound for $\rmJ_{k, M}$ from the Sturm bound of $\rmM_{k + 2j}$ with $0 \le j \le 3$, which allows us to determine Jacobi forms from their initial Fourier series expansions.  The methods in~\cite{Ra12-special-cycles} show how to find an explicit basis with integral Fourier series coefficients.  In Table~\ref{tab:jacobi-forms-fourier-expansions}, we give the initial Fourier series coefficients of such a basis $\phi_{k, i}$ ($0 \le i < \dim \rmJ_{k, M}$) of $\rmJ_{k, M}$ for $k \in \{4, 6, 8, 10\}$. Representatives of $r \in \ZZ^l$ modulo $2 M \ZZ^l$ and $\{\pm 1\}$ are given by
\begin{alignat*}{4}
  r_0 &= (0, 0, 0)
\text{,}
&
\quad
  r_1 &= (0, 0, 1)
\text{,}
&
\quad
  r_2 &= (0, 1, 0)
\text{,}
&
\quad
  r_3 &= (0, 1, 1)
\text{,}
\\
  r_4 &= (1, 0, 0)
\text{,}
&
  r_5 &= (1, 0, 1)
\text{,}
&
  r_6 &= (1, 1, 0)
\text{,}
&
  r_7 &= (1, 1, 1)
\text{.}
\end{alignat*}

\begin{table}[t]
{\tiny
\begin{tabular}{l@{\hspace{2em}}r@{\hspace{1em}}rr@{\hspace{1em}}rrrr@{\hspace{1em}}rrrrr}
\toprule
 &
  $\rmJ_{4, M}$ & $\rmJ_{6, M}$ &&
  $\rmJ_{8, M}$ &&&&
  $\rmJ_{10, M}$ &&&&
\\ 
 &
 $\phi_{4, 0}$ & $\phi_{6, 0}$ & $\phi_{6, 1}$ &
 $\phi_{8, 0}$ & $\phi_{8, 1}$ & $\phi_{8, 2}$ & $\phi_{8, 3}$ &
 $\phi_{10, 0}$ & $\phi_{10, 1}$ & $\phi_{10, 2}$ & $\phi_{10, 3}$ & $\phi_{10, 4}$
\\
\midrule
$(0, r_0)$ &
  $1$ & $1$ & $0$ & $1$ & $0$ & $0$ & $0$ & $1$ & $0$ & $0$ & $0$ & $0$ \\
$(1, r_0)$ &
  $26$ & $2$ & $8$ & $26$ & $40$ & $0$ & $0$ & $2$ & $8$ & $0$ & $0$ & $0$ \\
$(1, r_1)$ &
  $16$ & $-104$ & $-4$ & $0$ & $0$ & $2$ & $0$ & $6$ & $6$ & $10$ & $0$ & $0$ \\
$(1, r_2)$ &
  $16$ & $-104$ & $-4$ & $0$ & $0$ & $0$ & $2$ & $6$ & $6$ & $0$ & $10$ & $0$ \\
$(1, r_3)$ &
  $8$ & $20$ & $2$ & $84$ & $-10$ & $-1$ & $-1$ & $0$ & $1$ & $1$ & $1$ & $3$ \\
$(1, r_4)$ &
  $16$ & $-104$ & $-4$ & $264$ & $-36$ & $-2$ & $-2$ & $-104$ & $-14$ & $-10$ & $-10$ & $-10$ \\
$(1, r_5)$ &
  $8$ & $20$ & $2$ & $-48$ & $8$ & $0$ & $1$ & $-11$ & $-1$ & $0$ & $-1$ & $2$ \\
$(1, r_6)$ &
  $8$ & $20$ & $2$ & $-48$ & $8$ & $1$ & $0$ & $-11$ & $-1$ & $-1$ & $0$ & $2$ \\
$(1, r_7)$ &
  $2$ & $-16$ & $-1$ & $-4$ & $1$ & $0$ & $0$ & $0$ & $0$ & $0$ & $0$ & $-1$ \\
$(2, r_0)$ &
  $72$ & $-2064$ & $-48$ & $8760$ & $-48$ & $0$ & $0$ & $-28704$ & $-48$ & $0$ & $0$ & $1920$ \\
$(2, r_1)$ &
  $64$ & $-608$ & $16$ & $5632$ & $-256$ & $-24$ & $0$ & $-10616$ & $-56$ & $-72$ & $0$ & $-960$ \\
$(2, r_2)$ &
  $64$ & $-608$ & $16$ & $5632$ & $-256$ & $0$ & $-24$ & $-10616$ & $-56$ & $0$ & $-72$ & $-960$ \\
$(2, r_3)$ &
  $48$ & $-552$ & $-4$ & $440$ & $228$ & $10$ & $10$ & $-3904$ & $-114$ & $-82$ & $-82$ & $426$ \\
$(2, r_4)$ &
  $64$ & $-608$ & $16$ & $2464$ & $176$ & $24$ & $24$ & $-9824$ & $88$ & $72$ & $72$ & $-888$ \\
$(2, r_5)$ &
  $48$ & $-552$ & $-4$ & $1760$ & $48$ & $0$ & $-10$ & $-3002$ & $50$ & $0$ & $82$ & $508$ \\
$(2, r_6)$ &
  $48$ & $-552$ & $-4$ & $1760$ & $48$ & $-10$ & $0$ & $-3002$ & $50$ & $82$ & $0$ & $508$ \\
$(2, r_7)$ &
  $48$ & $-288$ & $0$ & $960$ & $-72$ & $0$ & $0$ & $-512$ & $16$ & $0$ & $0$ & $-256$ \\
$(3, r_0)$ &
  $144$ & $-4128$ & $96$ & $78576$ & $288$ & $0$ & $0$ & $-581696$ & $-32$ & $0$ & $0$ & $5888$ \\
$(3, r_1)$ &
  $112$ & $-3992$ & $4$ & $43008$ & $1024$ & $110$ & $0$ & $-298982$ & $282$ & $214$ & $0$ & $-4864$ \\
$(3, r_2)$ &
  $112$ & $-3992$ & $4$ & $43008$ & $1024$ & $0$ & $110$ & $-298982$ & $282$ & $0$ & $214$ & $-4864$ \\
$(3, r_3)$ &
  $112$ & $-3272$ & $-20$ & $32728$ & $-556$ & $-30$ & $-30$ & $-141184$ & $1094$ & $774$ & $774$ & $3794$ \\
$(3, r_4)$ &
  $112$ & $-3992$ & $4$ & $57528$ & $-956$ & $-110$ & $-110$ & $-301336$ & $-146$ & $-214$ & $-214$ & $-5078$ \\
$(3, r_5)$ &
  $112$ & $-3272$ & $-20$ & $28768$ & $-16$ & $0$ & $30$ & $-149698$ & $-454$ & $0$ & $-774$ & $3020$ \\
$(3, r_6)$ &
  $112$ & $-3272$ & $-20$ & $28768$ & $-16$ & $30$ & $0$ & $-149698$ & $-454$ & $-774$ & $0$ & $3020$ \\
$(3, r_7)$ &
  $50$ & $-1552$ & $15$ & $15836$ & $9$ & $0$ & $0$ & $-69120$ & $-128$ & $0$ & $0$ & $-2017$ \\
$(4, r_0)$ &
  $218$ & $-16270$ & $-64$ & $400586$ & $-3008$ & $0$ & $0$ & $-5013022$ & $704$ & $0$ & $0$ & $-27648$ \\
$(4, r_1)$ &
  $192$ & $-13344$ & $-80$ & $279040$ & $-768$ & $-200$ & $0$ & $-3114792$ & $-1384$ & $-1048$ & $0$ & $8896$ \\
$(4, r_2)$ &
  $192$ & $-13344$ & $-80$ & $279040$ & $-768$ & $0$ & $-200$ & $-3114792$ & $-1384$ & $0$ & $-1048$ & $8896$ \\
$(4, r_3)$ &
  $160$ & $-8816$ & $40$ & $190608$ & $-968$ & $-20$ & $-20$ & $-1869440$ & $-3756$ & $-2668$ & $-2668$ & $-1540$ \\
$(4, r_4)$ &
  $192$ & $-13344$ & $-80$ & $252640$ & $2832$ & $200$ & $200$ & $-3103264$ & $712$ & $1048$ & $1048$ & $9944$ \\
$(4, r_5)$ &
  $160$ & $-8816$ & $40$ & $187968$ & $-608$ & $0$ & $20$ & $-1840092$ & $1580$ & $0$ & $2668$ & $1128$ \\
$(4, r_6)$ &
  $160$ & $-8816$ & $40$ & $187968$ & $-608$ & $20$ & $0$ & $-1840092$ & $1580$ & $2668$ & $0$ & $1128$ \\
$(4, r_7)$ &
  $240$ & $-8352$ & $0$ & $124608$ & $792$ & $0$ & $0$ & $-1063424$ & $208$ & $0$ & $0$ & $-3328$ \\
\bottomrule \\
\end{tabular}}
\caption{Initial Fourier series expansions of Jacobi forms of weights~$4$, $6$, $8$, and $10$ and index~$M$.}
\label{tab:jacobi-forms-fourier-expansions}
\end{table}

Recall that $p\geq 5$ throughout the paper.  We have $\rmJ_{2, M} = \{ 0 \}$, and $\dim \rmJ_{4, M}=1$, but no nonzero element in $\rmJ_{4, M}$ has a $U_p$ congruence.

The space $\rmJ_{6, M}$ has dimension~$2$, and Theorem~\ref{thm:Jacobi-U_p} asserts that there are no $U_p$ congruences for $p \geq 7$.  If $p=5$, then we find that Jacobi forms that have $U_5$ congruences form a $1$-dimensional space which is spanned by the unique cusp form~$\phi_{6,1}$.

Theorem~\ref{thm:Jacobi-U_p} shows that forms in $\rmJ_{8, M}$, which has dimension~$4$, can have $U_p$ congruences only for $p \in \{5, 7\}$.  We find that the space of Jacobi forms exhibiting $U_p$ congruences has dimension~$2$ if $p=5$ and dimension~$1$ if $p=7$.  More specifically, forms with $U_5$ congruences are linear combinations of~$\phi_{8, 2}$ and~$\phi_{8,3}$, and the space of Jacobi forms with $U_7$ congruences is spanned by~$\phi_{8,1} + \phi_{8,2} + \phi_{8,3}$.

Finally, we consider the space $\rmJ_{10, M}$, which is $5$-dimensional. Theorem~\ref{thm:Jacobi-U_p} asserts that there are no $U_p$ congruences for $p>13$.  If $p \in \{5, 7, 11, 13\}$, then  the corresponding spaces of Jacobi forms with $U_p$ congruences have dimensions~$1$, $1$, $2$, and $4$, respectively.  They equal the span of
\begin{alignat*}{2}
&
  \phi_{10, 1} + 4 \phi_{10, 2} + 4 \phi_{10, 3} + \phi_{10, 4}
\quad
&&
  \text{for $p = 5$}
\text{;}
\\
&
  \phi_{10, 1} + 5 \phi_{10, 2} + 5 \phi_{10, 3} + 4 \phi_{10, 4}
\quad
&&
  \text{for $p = 7$}
\text{;}
\\
&
  \phi_{10, 1} + 9 \phi_{10, 4}
\;\,
\text{and}\quad
  \phi_{10, 2} + \phi_{10, 3} + 8 \phi_{10, 4}
\quad
&&
  \text{for $p = 11$}
\text{;}
\\
&
  \phi_{10, 1}
\text{,}\;\,
  \phi_{10, 2}
\text{,}\;\,
  \phi_{10, 3}
\text{,}\;\,
\text{and}\quad
 \phi_{10, 4}
\quad
&&
  \text{for $p = 13$}
\text{.}
\end{alignat*}

In the case $k = 10$ and $p = 11$, we apply Theorem~\ref{thm:Jacobi-U_p} to verify that $\bbL^{\frac{3 p + l}{2} + 1 - k}(\phi)$ is congruent to a form of weight $3 p + 2 + l - k$.   In the cases that $\phi \in \rmJ_{k,M}(\ZZ_{(p)})$ with $k = 6, 8, 10$ and $p=5$, $k = 8, 10$ and $p=7$, and $k=10$ and $p =13$, we establish the above congruences by checking that $\bbL^{p - 1} \phi \equiv \phi\pmod{p}$.  These later computations are more expensive, since the filtration of $\bbL^{p - 1} \phi$ could be as large as $k+(p+1)(p-1)$.  For example, in the case $k = 10$, $p = 13$ the Sturm bound shows that it suffices to check Fourier series coefficients for all $n < 17$.  In comparison, if our Theorem~\ref{thm:Jacobi-U_p} had been applicable, Fourier series coefficients up to $n < 5$ would have sufficed.

\needspace{5\baselineskip}
\section{Siegel modular forms}
\label{sec:Siegelmodularforms}
 
\subsection{Preliminaries}
\label{sec:Siegelpreliminaries}

Let $\HS_g$ be the Siegel upper half space of degree~$g$, $Z\in\HS_g$ be a typical variable, and $\Sp{g}(\ZZ)$ be the symplectic group of degree~$g$ over the integers.  For a ring $R\subset\CC$, let $\rmM^{(g)}_k (R)$ be the vector space of Siegel modular forms of degree~$g$, weight~$k$, and with coefficients in $R$
(for details on Siegel modular forms, see for example Freitag \cite{Frei_1} or Klingen \cite{Klingen}).  We suppress $R$ if $R=\CC$.  Recall that $\Phi\in\rmM^{(g)}_k$ satisfies the transformation law
\begin{gather*}
  \Phi\big( (A Z + B) (C Z + D)^{-1} \big)
=
  \det(C Z + D)^{k}\, \Phi(Z)
\end{gather*}
for all $\left(\begin{smallmatrix} A & B \\ C & D \end{smallmatrix}\right) \in \Sp{g}(\ZZ)$.  Furthermore, $\Phi$ has a Fourier series expansion of the form
\begin{gather*}
  \sum_{T=\rT{T}\geq 0} \hspace{-.5em}
  c(\Phi; T)\, e^{2\pi i\, \tr(T Z) }
\text{,}
\end{gather*}
where the sum is over all symmetric, semi-positive definite, and half-integral $g\times g$ matrices with integral diagonal entries.  We always denote the Fourier series coefficients of a Siegel modular form~$\Phi$ by $c(\Phi;T)$.   Write $Z=\left(\begin{smallmatrix}\tau & z\\ \rT{z} & W\end{smallmatrix}\right)\in\HS_g$, 
where $\tau\in\HH$, $z\in\CC^{g-1}$, and $W\in\HH_{g-1}$ to find the Fourier-Jacobi expansion:
\begin{gather*}
  \Phi(Z)
=
  \Phi(\tau,z,W)
=
  \sum_{M=\rT{M}\geq 0} \hspace{-.5em}
  \Phi_{M}(\tau,z)\, e^{2\pi i\,\tr(MW)}
\text{,}
\end{gather*}
\noindent
where the sum is over all symmetric, semi-positive definite, and half-integral $g-1\times g-1$ matrices with integral diagonal entries, and where $\Phi_{M}\in\rmJ_{k, M}$ are Jacobi forms on $\HS\times\CC^{g-1}$.



\subsection{Siegel modular forms modulo~$p$}
\label{sec:modularformsmodp}
Recall that $p\geq 5$ is a prime.  Note that  B\"ocherer and Nagaoka~\cite{Bo-Na-MathAnn07} require that $p\geq g+3$, and whenever we apply their results we also assume that $p\geq g+3$.  For Siegel modular forms $\Phi(Z) = \sum c(\Phi;T)\, e^{2\pi i \, \tr(TZ)}$ and $\Psi(Z) = \sum c(\Psi;T)\, e^{2\pi i \, \tr(TZ)}$  with coefficients in $\ZZ_{(p)}$, we write $\Phi\equiv \Psi \pmod{p}$ when 
$c(\Phi;T)\equiv c(\Psi;T) \pmod{p}$ for all $T$.  Set
\begin{gather}
  \rmM_k^{(g)}(\bbF_p)
:=
  \big\{\widetilde{\Phi}\, :\, \Phi\in \rmM_k^{(g)}(\ZZ_{(p)}) \big\}
\text{,}
\end{gather}
where $\widetilde{\Phi}(Z) = \sum \widetilde{c}(\Phi;T)\, e^{2\pi i \,\tr(TZ)}$ (with $\wtd{c}(\Phi;T) \in \bbF_p$) denotes the reduction modulo $p$ of $\Phi$. 

We now introduce the weight filtration for Siegel modular forms modulo~$p$.  If $\Phi \in \rmM^{(g)}_k(\ZZ_{(p)})$ or $\Phi \in \rmM^{(g)}_k(\bbF_p)$, then
\begin{gather*}
  \omega(\Phi)
:=
  \inf \big\{ k' \in \ZZ \,:\,
              \exists_{ \Psi \in \rmM^{(g)}_{k'}(\ZZ_{(p)}) }\, \Psi \equiv \Phi \pmod{p} \big\}
\text{.}
\end{gather*}
 
We extend Proposition~\ref{prop:jacobi_weight_congruence} to Siegel modular forms.
\begin{proposition}
\label{prop:siegel-weight-congruence-homogeous-components}
Let $(\Phi_k)_k$ be a finite family of Siegel modular forms $\Phi_k \in \rmM^{(g)}_k(\ZZ_{(p)})$.  If $\sum_k \Phi_k \equiv 0 \pmod{p}$, then for all $a \in \ZZ \slashdiv (p - 1) \ZZ$ we have
\begin{gather*}
  \sum_{k \in a + (p - 1) \ZZ} \hspace{-.5em}
  \Phi_k
\equiv
  0 \pmod{p}
\text{.}
\end{gather*}
\end{proposition}
\begin{proof}
Consider the Fourier-Jacobi coefficients $(\Phi_k)_M$ of $\Phi_k$ to find that the desired congruence holds if and only if
\begin{gather*}
  \sum_{k \in a + (p - 1) \ZZ} \hspace{-.5em}
  (\Phi_k)_M
\equiv
  0 \pmod{p}
\text{,}
\end{gather*}
holds for all $M$. 
%
Thus, Proposition~\ref{prop:jacobi_weight_congruence} in combination with Remark~\ref{rm:semi-definite-indices} yields the claim.
\end{proof}

Proposition~\ref{prop:siegel-weight-congruence-homogeous-components} includes the following special case, which is due to~Ichikawa~\cite{Ichi-MathAnn08} (see also B\"ocherer-Nagaoka~\cite{Bo-Na-Abh10}) if $p \ge g + 3$.  
\begin{corollary}
\label{cor:siegel-weight-congruence}
Suppose that $0 \not \equiv \Phi \equiv \Psi \pmod{p}$ for $\Phi \in \rmM^{(g)}_k(\ZZ_{(p)})$ and $\Psi \in \rmM^{(g)}_{k'}(\ZZ_{(p)})$.  Then $k \equiv k' \pmod{(p - 1)}$.  In particular, $\omega(\Phi)$ is congruent modulo~$p - 1$ to the weight of~$\Phi$.
\end{corollary}

B\"ocherer and Nagaoka~\cite{Bo-Na-MathAnn07} establish the existence of Siegel modular forms which are congruent to~$1$ modulo~$p$, which is an important ingredient in determining the structure of $\rmM^{(g)}_\bullet(\bbF_p)$.
\begin{theorem}[\cite{Bo-Na-MathAnn07}]
\label{thm:existence-of-1-mod-siegel-forms}
Let $p\geq g+3$.  There exists a $\Psi_{p-1} \in \rmM^{(g)}_{p - 1}(\ZZ)$ such that $\Psi_{p-1} \equiv 1 \pmod{p}$.
\end{theorem}

\subsection{Structure of $\rmM^{(g)}_\bullet(\bbF_p)$}
\label{sec:structureofmodularformsmodp}


In this section, we study the implications of Proposition~\ref{prop:siegel-weight-congruence-homogeous-components} for the graded rings of Siegel modular forms.


\begin{proposition}
\label{prop:structure-of-siegel-saturation-of-ideal}
For all primes~$p$, the ideal $C$ in~\eqref{eq:siegel-reduction-isomorphism} is saturated with respect to $(p)$.  That is, if $P \in C$ and $p \isdiv P$, then also $p^{-1} P \in C$.  In particular, if  $P\in C$ is homogeneous and if $\Phi_i$ are generators of $\rmM^{(g)}_k(\ZZ_{(p)})$ such that $P(\Phi_1, \ldots, \Phi_n) \equiv 0 \pmod{p}$, then $P(x_1, \ldots, x_n) \equiv 0 \pmod{p}$.
\end{proposition}
\begin{proof}
Let~$P \in \ZZ_{(p)}[x_1, \ldots, x_n]$ be a polynomial whose coefficients are divisible by~$p$ and which satisfies $P(\Phi_1, \ldots, \Phi_n) = 0$.  Then $(p^{-1}P )(\Phi_1, \ldots, \Phi_n) = p^{-1}\big(P(\Phi_1, \ldots, \Phi_n)\big)= 0$, which gives the claim.
\end{proof}


We now prove Theorem~\ref{thm:structure-of-siegel-mod-p-rings}, which determines the ring structure of $\rmM^{(g)}_\bullet(\bbF_p)$.

\begin{proof}[Proof of Theorem~\ref{thm:structure-of-siegel-mod-p-rings}]
Let $P(x_1, \ldots, x_n)$ be a polynomial such that $P(\Phi_1, \ldots, \Phi_n) \equiv 0 \pmod{p}$.  We need to show that $P$ is divisible by $1-B$ in $\bbF_p[x_1, \ldots, x_n]$.

By Proposition~\ref{prop:siegel-weight-congruence-homogeous-components}, we may assume that there exists an $a \in \ZZ$ such that $P = \sum_{k \equiv a \pmod{p - 1}} P_k$ for homogeneous components $P_k$ of degree~$k$.  Let $k \in \ZZ$ be minimal subject to the condition $P_k \not\equiv 0 \pmod{p}$.  Replacing $P$ by $Q:= P - (1 - B) P_k$ yields a polynomial whose homogeneous components~$Q_{k'}$ vanish (modulo~$p$) for $k' < k + p - 1$.
By iterating this process, we may assume that $P \equiv P_k \pmod{p}$ for some $k$.  Then $P(\Phi_1, \ldots, \Phi_n) \equiv P_k(\Phi_1, \ldots, \Phi_n) \equiv 0 \pmod{p}$ for this $k$, and Proposition~\ref{prop:structure-of-siegel-saturation-of-ideal} implies that $P(x_1, \ldots, x_n) \equiv P_k(x_1,\ldots x_n)\equiv 0 \pmod{p}$, which completes the proof.
\end{proof}
\begin{proof}[{Proof of Corollary~\ref{cor:siegel-characterization-of-maximal-filtration}}]
Let $\Phi'\in \rmM^{(g)}_{k'}(\ZZ_{(p)})$ with $k' < k$ and such that $\Phi \equiv \Phi' \pmod{p}$, and suppose that $\Phi'$ corresponds to the polynomial $A'$ under~\eqref{eq:siegel-reduction-isomorphism}.  We employ Theorem~\ref{thm:structure-of-siegel-mod-p-rings} to find that 
\begin{gather}
\label{eq:B-divides-A}
(1 - B) P \equiv (A - A') \pmod{p}
\end{gather}
for some (not necessarily homogeneous) $P\in \bbF_p [x_1, \ldots, x_n]$.  Write $P = \sum P_l$ for polynomials $P_l$ which are homogeneous of degree~$l$ with respect to the weighted grading of $\bbF_p [x_1, \ldots, x_n]$.  Comparing the left and right hand side of \eqref{eq:B-divides-A} shows that
$B P_{k - (p - 1)} \equiv A\pmod{p}$.  Hence $B$ divides $A$.  On the other hand, if $B$ divides $A$, then $\Phi=\Psi_{p-1}\Phi'$ for some $\Phi'\in \rmM^{(g)}_{k'}(\ZZ_{(p)})$ and $\omega(\Phi) < k$.
%
%
\end{proof}

We end this section with two Lemmas that are needed in the proof of Proposition~\ref{prop:filtration-of-D-operator}.

\begin{lemma}
\label{lem:existence-of-relatively-prime-polynomial}
Let $p \ge g + 3$.  For every sufficiently large~$k$ there exists a
$\Phi \in \rmM^{(g)}_k(\ZZ_{(p)})$ such that $P_\Phi$ and $B$ are coprime in $\ZZ_{(p)}[x_1, \ldots, x_n] \slashdiv C$, where $P_\Phi$ is the polynomial corresponding to $\Phi$ under the isomorphism in~\eqref{eq:siegel-reduction-isomorphism}.
\end{lemma}
\begin{proof}
Theorem~2 of~\cite{SwD-l-adic} treats the case~$g = 1$.  Assume that $g > 1$.  It suffices to show the statement in $R_\bullet := \big( \ZZ_{(p)}[x_1, \ldots, x_n] \slashdiv C \big) \otimes \bbF_p$.  There are finitely many irreducible divisors $D_i \in R_\bullet$ of $B \pmod{p}$.  Write $k_i$ for their degrees, and $H(k)$ for the Hilbert polynomial of~$R_\bullet$, whose degree equals~$\frac{g(g + 1)}{2} > 1$.  The set of elements in $R_k$ which are divisible by $D_i$ equals $D_i R_{k - k_i}$.  Hence the set of elements in $R_k$ which are not coprime to~$B$ is contained in the union of hyperplanes $D_i R_{k - k_i}$.  Thus, we are reduced to showing that $p^{H(k)} - \sum_i p^{H(k - k_i)} > 0$ for large enough~$k$.  This is clear, since the degree of $H(k)$ is greater than~$1$, and therefore $H(k) - H(k - k_i) \rightarrow \infty$ as $k \rightarrow \infty$.
\end{proof}

\begin{lemma}
\label{la:MFp-integral-domain}
Let $p$ be any prime.  Let $\Phi, \Psi \in \rmM^{(g)}_\bullet(\ZZ_{(p)})$ such that $\Phi, \Psi \not \equiv 0 \pmod{p}$.  Then $\Phi \Psi \not\equiv 0 \pmod{p}$.  In particular, $\rmM^{(g)}_\bullet(\bbF_p)$ is an integral domain.
\end{lemma}
\begin{proof}
Consider the order $T = (T_{ij}) < T' = (T'_{ij})$ defined by
\begin{gather*}
  \big( \exists_{1 \le i \le g} :
        T_{ii} < T'_{ii} \wedge (\forall_{ 1 \le j < i}\, T_{jj} = T'_{jj} ) \big)
  \;\vee\;
\\
  \Big( (\forall_{1 \le i \le g}\, T_{ii} = T'_{ii}) \,\wedge\,
        \big(\exists_{1 \le i,j \le g} :
             T_{ij} < T'_{ij} \,\wedge\,
             \forall_{\substack{ 1 \le i' < i \\ 1 \le j' \le g }}\, T_{i' j'} = T'_{i' j'} \,\wedge\,
             \forall_{1 \le j' < j}  T_{i j'}\, = T'_{i j'}  \big)
       \Big)
\text{.}
\end{gather*}
Choose minimal indices $T_\Phi$ and $T_\Psi$ with respect to the order subject to the condition that $c(\Phi; T_\Phi) \not\equiv 0\pmod{p}$ and $c(\Psi; T_\Psi) \not\equiv 0 \pmod{p}$, respectively.  If $T_{\Phi} + T_{\Psi}=T_1 + T_2 $, then either $T_{\Phi}=T_1$ and $T_{\Psi}=T_2$, or one of the following holds:
\begin{gather*}
  T_1 < T_{\Phi}
\;\,\text{or}\;\,
  T_1 < T_{\Psi}
\;\,\text{or}\;\,
  T_2 < T_{\Phi}
\;\,\text{or}\;\,
  T_2 < T_{\Psi}
\text{.}
\end{gather*}
We find that 
\begin{gather*}
  c(\Phi \Psi; T_\Phi + T_\Psi)
\equiv
  c(\Phi; T_\Phi)
  c(\Psi; T_\Psi)
\not\equiv
  0 \pmod{p}
	\text{,}
\end{gather*}
which gives the claim.
\end{proof}

\subsection{Theta cycles and $U(p)$-congruences}
\label{sec:ThetacyclesandUcongruences}

In this section, we explore theta cycles of Siegel modular forms, and we prove Theorem~\ref{thm:Siegel-U_p}.  From the introduction, recall the generalized theta operator
\begin{gather*}
\mathbb{D}:=(2\pi i)^{-g}\det \partial_Z
\text{.}
\end{gather*}
B\"ocherer and Nagaoka \cite{Bo-Na-MathAnn07} prove that if $\Phi \in \rmM^{(g)}_k(\ZZ_{(p)})$, then 
\begin{gather}
\label{eq:action-of-D-operator}
\mathbb{D}(\Phi)\in   \rmM_{k+p+1}^{(g)}(\bbF_p).
\end{gather}  
Our following result on the filtration of $\mathbb{D}(\Phi)$ generalizes a classical result on 
elliptic modular forms (see \cite{Serre-p-adic, SwD-l-adic}) and also Proposition~4 of~\cite{C-C-R-Siegel} to the case of Siegel modular forms of degree $g>2$.
\begin{proposition}
\label{prop:filtration-of-D-operator}
Let $p\geq g+3$.  Let $\Phi \in \rmM^{(g)}_k(\ZZ_{(p)})$ with $2k\geq g$ 
and suppose that there exists a Fourier-Jacobi coefficient
$\Phi_M$ of $\Phi$ such that $p\,\nmid \,\det(2M)$ and $\omega\big(\Phi_M\big)=\omega(\Phi)$.  Then
\begin{equation*}
\omega\big(\mathbb{D}(\Phi)\big)\leq \omega(\Phi)+p+1,
\end{equation*}
\noindent
with equality if and only if $p\,\nmid \, \left(2\omega(\Phi)-g+1\right)$.
\end{proposition}

\begin{proof}
We proceed exactly as in \cite{C-C-R-Siegel}, and we assume that $\omega(\Phi)=k$.  Consider the Fourier-Jacobi expansion
\begin{gather*}
  \Phi(Z)
=
  \Phi(\tau,z,W)
=
  \sum_{M=\rT{M}\geq 0} \hspace{-.5em}
  \Phi_{M}(\tau,z)\, e^{2\pi i\,\tr(MW)}
\text{,}
\end{gather*}
\noindent
where $\Phi_{M}\in\rmJ_{k, M}(\ZZ_{(p)})$.  Then
\begin{gather*}
  \mathbb{D}(\Phi)
=
  \sum_{M=\rT{M}\geq 0} \hspace{-.3em}
  \bbL_M(\Phi_{M}(\tau,z))\, e^{2\pi i\,\tr(MW)}.
\end{gather*}
Let $\Phi_M$ such that $p\,\nmid \,\det(2M)$ and $\omega\big(\Phi_M\big)=\omega(\Phi)$.  If $p \,\nmid \, \left(2k-g+1\right)$, then Proposition~\ref{prop: filtration-of-heat-application} implies that $\omega\big(\bbL_M(\Phi_{M})\big) = k+p+1$.  Moreover, for each $M$ we have 
\begin{gather*}
  \omega\big(\bbL_M\left(\Phi_M\right)\big)
\leq
  \omega\big(\mathbb{D}(\Phi)\big)
\underset{(\ref{eq:action-of-D-operator})}{\leq}
  k+p+1
\end{gather*}
and hence $\omega\big(\mathbb{D}(\Phi)\big) = k+p+1$. 

Now assume that $p \mid (2k-g+1)$. Lemma~\ref{lem:existence-of-relatively-prime-polynomial} shows that there exists an $\Upsilon \in \rmM^{(g)}_{k'}(\ZZ_{(p)})$ with $2k'\geq g$ such that $\omega(\Upsilon)=k'$, $p \nisdiv\frac{(2k')!}{(2k'-g)!}$, and $P_{\Upsilon}$ is relatively prime to $B$, where $P_{\Upsilon}$ corresponds to $\Upsilon$ under the isomorphism in Theorem~\ref{thm:structure-of-siegel-mod-p-rings}.

Eholzer and Ibukiyama~\cite{Eho-Ibu} establish a Rankin-Cohen bracket of Siegel modular forms of degree $g$.  We employ a slight extension of a special case of~\cite{Eho-Ibu}, which is due to B\"ocherer and Nagaoka~\cite{Bo-Na-MathAnn07}. Let us introduce necessary notation: For $0\leq \alpha\leq g$ let $P_{\alpha}(R,R')$ be the polynomial defined by the equation
\begin{gather*}
  \det(R+\lambda R')
=
  \sum_{\alpha=0}^g
  P_{\alpha}(R,R')\, \lambda^{\alpha}
\text{,}
\end{gather*}
where the variables $R$ and $R'$ are symmetric $n\times n$ matrices.  Set
\begin{gather*}
\mathcal{Q}_{k,k'}^{(g)}(R,R'):=\sum_{\alpha=0}^g(-1)^{\alpha}\alpha!(g-\alpha)!{2k'-\alpha \choose g-\alpha}{2k-g+\alpha \choose \alpha}P_{\alpha}(R,R')
\text{}
\end{gather*}
and
\begin{gather*}
\widetilde{D}_{k,k'}^{(g)}:=(2\pi i)^{-g}\mathcal{Q}_{k,k'}^{(g)}(\partial_{Z_1}, \partial_{Z_2})
\text{.}
\end{gather*}
Corollary~2 of~\cite{Bo-Na-MathAnn07} asserts that  
\begin{gather*}
\{\Phi,\Upsilon\}(Z):=\widetilde{D}_{k,k'}^{(g)}(\Phi(Z_1),\Upsilon(Z_2))\big|_{Z_1=Z_2=Z}\in\rmM^{(g)}_{k+k'+2}(\ZZ_{(p)})
\text{,}
\end{gather*}
and a direct computation (observing that $p \mid (2k-g+1)$) shows that
\begin{gather*}
\{\Phi,\Upsilon\} \equiv \frac{(2k')!}{(2k'-g)!} \mathbb{D}(\Phi)\Upsilon \pmod{p}.
\end{gather*}

\noindent
If $\omega\big(\mathbb{D}(\Phi)\big)=k+p+1$, then there exists a $\Xi\in \rmM^{(g)}_{k+p+1}(\ZZ_{(p)})$ such that $\omega(\Xi)=k+p+1$ and $\mathbb{D}(\Phi)\equiv \Xi \pmod{p}$.  Let $P_{\Xi}$ be the polynomial corresponding to $\Xi$ under the isomorphism in Theorem~\ref{thm:structure-of-siegel-mod-p-rings}.  Note that $B$ does not divide $P_{\Xi}P_{\Upsilon}$, since $B$ does not divide $P_{\Xi}$ and $B$ is relatively prime to $P_{\Upsilon}$.  Lemma~\ref{la:MFp-integral-domain} shows that $\Xi\Upsilon\not\equiv0\pmod{p}$ and Corollary~\ref{cor:siegel-characterization-of-maximal-filtration} implies that $\omega\big(\Xi\Upsilon\big)=k'+k+p+1$.  We obtain the contradiction
\begin{equation*}
k'+k+2\geq \omega\big(\{\Phi,\Upsilon\}\big)=
\omega\big(\mathbb{D}(\Phi)\Upsilon\big)=\omega\big(\Xi\Upsilon\big)=k'+k+p+1.
\end{equation*}
Hence $\omega\big(\mathbb{D}(\Phi)\big)<k+p+1$, which completes the proof.
\end{proof}

\begin{proof}[Proof of Theorem~\ref{thm:Siegel-U_p}]
We apply Corollary~\ref{cor:siegel-weight-congruence} and Proposition~\ref{prop:filtration-of-D-operator} to study generalized theta cycles.  Theorem~\ref{thm:Siegel-U_p} follows then analogous to Theorem~\ref{thm:Jacobi-U_p}.
\end{proof}

\subsection{Example}
\label{sec:Siegelexample}

In this final section, we apply Theorem~\ref{thm:Siegel-U_p} to determine all $U(p)$ congruences of the Schottky form $J_4$.  Recall that $J_4$ is the unique Siegel cusp form of weight~$8$ and degree~$4$ with integral Fourier series coefficients, and normalized such that the content of its Fourier series coefficients is~$1$.  Theorem~\ref{thm:Siegel-U_p} implies that $J_4\,\big|\,U(p) \not\equiv 0\pmod{p}$ if $p > 9 = 2 \cdot 8 - 4 - 3$.  It remains to discuss the cases $p = 5$ and $p = 7$.  If $p = 5$, then $J_4\,\big|\,U(5) \not\equiv 0\pmod{5}$, since $c(J_4;T)=-1$, where
\begin{gather*}
 T= \begin{pmatrix}
  2 & 1 & 0 & 1 \\
  1 & 2 & 0 & 0 \\
  0 & 0 & 2 & 1 \\
  1 & 0 & 1 & 2
  \end{pmatrix}
\end{gather*}
with $\det T=5$ (see p. 218 of \cite{P-Y-Japan07}).

In order to show that $J_4$ has a $U(p)$ congruence if $p = 7$, we consider $J_4$ as the Duke-Imamo\u glu-Ikeda lift of the unique elliptic cusp form~$\delta$ (normalized by $c(\delta;\, 1)=1$) of weight~$\frac{13}{2}$ and level~$4$.  Note that $\delta$ has integral Fourier series coefficients $c(\delta;\, n)$, and~\cite{K-Z-Invent81} provides a table of~$c(\delta;\, n)$ with $n \leq 149$.  Kohnen's~\cite{Ko-MathAnn2002} formula for the Fourier series coefficients of Duke-Imamo\u glu-Ikeda lifts asserts that
\begin{gather*}
  c(J_4; T)
=
  \sum_{a \isdiv f_T} a^6 \phi(a; T)\,
  c\Big( \delta;\, D_{T, 0} \Big(\frac{f_T}{a}\Big)^2 \Big)
\text{,}
\end{gather*}
where $\phi(a, T) \in \ZZ$ is a local invariant of~$T$, and $\det(T) = D_{T, 0} f_T^2$ for  square free $D_{T, 0}$.  Observe that $c(\delta, n) \equiv 0 \pmod{7}$ whenever $7 \isdiv n$ implies that $c(J_4; T) \equiv 0 \pmod{7}$ whenever $7 \isdiv \det(T)$.  Thus, it suffices to verify that $\delta$ has a $U(7)$ congruence.

If $g=1$, then $\theta:=\mathbb{D}$ is the usual theta operator for elliptic modular forms.  We need to argue that $\delta \equiv \theta^6 \delta \pmod{7}$.  Note that $\theta^6 \delta$ is congruent to a modular form of weight $\frac{13}{2}+48$ and level $4$, and the Sturm bound is~$28$ (see for example Theorem~2.58 of~\cite{Ono}).  The table in~\cite{K-Z-Invent81} shows that the first~$28$ coefficients of~$\delta$ exhibit a $U(7)$ congruence, and we conclude that $J_4\,\big|\,U(7)\equiv 0\pmod{7}$.

\begin{remark}
Recall that the Schottky $J_4$ is the difference of the theta series of degree~$4$ attached to the unimodular lattices $E_8 \oplus E_8$ and $D_{16}$, respectively (see~\cite{Igusa-Schottky81}).  In this context, our example can be interpreted as follows:  For all primes~$p \ge 5$, $p \ne 7$, there exists at least one lattice~$\Lambda$ of dimension four whose discriminant is divisible by~$p$ and which does not occur with the same multiplicities in $E_8 \oplus E_8$ and $D_{16}$.  The case $p = 7$ is special.  The number of $4$\nbd dimensional sublattices $\Lambda$ in $E_8 \oplus E_8$ and $D_{16}$ with discriminant divisible by~$7$ coincides modulo~$7$.

Comparing elliptic theta series attached to these lattices shows that the number of vectors of length~$n$ in $E_8 \oplus E_8$ and $D_{16}$ is the same for all non\nbd negative integers~$n$.
\end{remark}

\bibliographystyle{plain}
\bibliography{Lit}

\begin{thebibliography}{10}

\bibitem{Ahl-Ono-Comp05}
S.~Ahlgren and K.~Ono.
\newblock Arithmetic of singular moduli and class polynomials.
\newblock {\em Compos.\ {M}ath.}, {\bf 141}(2):293--312, 2005.

\bibitem{Bo-Na-MathAnn07}
S.~B{\"o}cherer and S.~Nagaoka.
\newblock On mod {$p$} properties of {S}iegel modular forms.
\newblock {\em Math.\ {A}nn.}, {\bf 338}(2):421--433, 2007.

\bibitem{Bo-Na-Abh10}
S.~B{\"o}cherer and S.~Nagaoka.
\newblock Congruences for {S}iegel modular forms and their weights.
\newblock {\em Abh.\ {M}ath.\ {S}emin.\ {U}niv.\ {H}ambg}, {\bf
  80}(2):227--231, 2010.

\bibitem{C-C-R-Siegel}
D.~Choi, Y.~Choie, and O.~Richter.
\newblock Congruences for {S}iegel modular forms.
\newblock {\em Ann.\ {I}nst.\ {F}ourier ({G}renoble)}, {\bf 61}(4):1455--1466,
  2011.

\bibitem{Ch-Kim}
Y.~Choie and H.~Kim.
\newblock Differential operators on {J}acobi forms of several variables.
\newblock {\em J.\ {N}umber {T}heory}, {\bf 82}:140--163, 2000.

\bibitem{Eho-Ibu}
W.~Eholzer and T.~Ibukiyama.
\newblock Rankin-{C}ohen type differential operators for {S}iegel modular
  forms.
\newblock {\em Internat.\ {J}.\ {M}ath.}, {\bf 9}(4):443--463, 1998.

\bibitem{EZ}
M.~Eichler and D.~Zagier.
\newblock {\em The theory of {J}acobi forms}.
\newblock Birkh\"auser, Boston, 1985.

\bibitem{E-O-Y-IMRN05}
N.~Elkies, K.~Ono, and T.~Yang.
\newblock Reduction of {CM} elliptic curves and modular function congruences.
\newblock {\em Internat.\ {M}ath.\ {R}es.\ {N}otices}, {\bf
  2005}(44):2695--2707.

\bibitem{Frei_1}
E.~Freitag.
\newblock {\em Siegelsche {M}odulfunktionen}.
\newblock Springer, {B}erlin, {H}eidelberg, {N}ew {Y}ork, 1983.

\bibitem{Gue-PAMS2007}
P.~Guerzhoy.
\newblock On ${U}(p)$-congruences.
\newblock {\em Proc.\ {A}mer.\ {M}ath.\ {S}oc.}, {\bf 135}(9):2743--2746, 2007.

\bibitem{Harder-1-2-3}
G.~Harder.
\newblock A congruence between a {S}iegel and an elliptic modular form.
\newblock In {\em The 1-2-3 of modular forms}, Universitext, pages 247--262.
  Springer, Berlin, 2008.

\bibitem{Ichi-MathAnn08}
T.~Ichikawa.
\newblock Congruences between {S}iegel modular forms.
\newblock {\em Math.\ {A}nn.}, {\bf 342}(3):527--532, 2008.

\bibitem{Igusa-Schottky81}
J.~Igusa.
\newblock Schottky's invariant and quadratic forms.
\newblock In {\em E. {B}. {C}hristoffel ({A}achen/{M}onschau, 1979)}, pages
  352--362. Birkh\"auser, Basel, 1981.

\bibitem{Joch-1982}
N.~Jochnowitz.
\newblock A study of the local components of the {H}ecke algebra mod {$l$}.
\newblock {\em Trans.\ {A}mer.\ {M}ath.\ {S}oc.}, {\bf 270}(1):253--267, 1982.

\bibitem{Kats-MathZ08}
H.~Katsurada.
\newblock Congruence of {S}iegel modular forms and special values of their
  standard zeta functions.
\newblock {\em Math.\ {Z}.}, {\bf 259}(1):97--111, 2008.

\bibitem{Klingen}
H.~Klingen.
\newblock {\em Introductory lectures on {S}iegel modular forms}, volume~{\bf
  20} of {\em \rm {C}ambridge {S}tudies in {A}dvanced {M}athematics}.
\newblock Cambridge {U}niversity {P}ress, 1990.

\bibitem{Ko-MathAnn2002}
W.~Kohnen.
\newblock Lifting modular forms of half-integral weight to {S}iegel modular
  forms of even genus.
\newblock {\em Math.\ {A}nn.}, {\bf 322}(4):787--809, 2002.

\bibitem{K-Z-Invent81}
W.~Kohnen and D.~Zagier.
\newblock Values of {$L$}-series of modular forms at the center of the critical
  strip.
\newblock {\em Invent.\ {M}ath.}, {\bf 64}(2):175--198, 1981.

\bibitem{M-M-London10}
G.~Mason and C.~Marks.
\newblock Structure of the module of vector-valued modular forms.
\newblock {\em J.\ {London} {M}ath.\ {S}oc. (2)}, {\bf 82}(1):32--48, 2010.

\bibitem{Na-MathZ00}
S.~Nagaoka.
\newblock Note on mod {$p$} {S}iegel modular forms.
\newblock {\em Math.\ {Z}.}, {\bf 235}(2):405--420, 2000.

\bibitem{Na-MathZ05}
S.~Nagaoka.
\newblock Note on mod {$p$} {S}iegel modular forms {II}.
\newblock {\em Math.\ {Z}.}, {\bf 251}(4):821--826, 2005.

\bibitem{Ono}
K.~Ono.
\newblock {\em The web of modularity: {A}rithmetic of the coefficients of
  modular forms and {$q$}-series}, volume 102 of {\em {CBMS} {R}egional
  {C}onference {S}eries in {M}athematics}.
\newblock Published for the {C}onference {B}oard of the {M}athematical
  {S}ciences, {W}ashington, {DC}, 2004.

\bibitem{P-Y-Japan07}
C.~Poor and D.~Yuen.
\newblock Computations of spaces of {S}iegel modular cusp forms.
\newblock {\em J.\ {M}ath.\ {S}oc.\ {J}apan}, {\bf 59}(1):185--222, 2007.

\bibitem{Ra12-special-cycles}
M.~Raum.
\newblock {Computing Jacobi Forms and Linear Equivalences of Special Divisors},
  2012.
\newblock arXiv:1212.1834.

\bibitem{crit}
O.~Richter.
\newblock On congruences of {J}acobi forms.
\newblock {\em Proc.\ {A}mer.\ {M}ath.\ {S}oc.}, {\bf 136}(8):2729--2734, 2008.

\bibitem{heat}
O.~Richter.
\newblock The action of the heat operator on {J}acobi forms.
\newblock {\em Proc.\ {A}mer.\ {M}ath.\ {S}oc.}, {\bf 137}(3):869--875, 2009.

\bibitem{Serre-p-adic}
J-P. Serre.
\newblock {\em {\it {F}ormes modulaires et fonctions zeta $p$-adiques}, {\rm
  in: Modular functions of one variable {III}}}, pages 191--268.
\newblock \rm Lecture {N}otes in {M}ath. {\bf 350}. Springer, 1973.

\bibitem{Sko-Weil}
N-P. Skoruppa.
\newblock Jacobi forms of critical weight and {W}eil representations.
\newblock In {\em Modular forms on {S}chiermonnikoog}, pages 239--266.
  Cambridge {Univ}.\ Press, Cambridge, 2008.

\bibitem{Sof-JNT97}
A.~Sofer.
\newblock $p$-adic aspects of {J}acobi forms.
\newblock {\em J.\ {N}umber {T}heory}, {\bf 63}(2):191--202, 1997.

\bibitem{sage57}
W.\thinspace{}A. Stein et~al.
\newblock {\em {S}age {M}athematics {S}oftware ({V}ersion 5.7)}.
\newblock The Sage Development Team, 2013.
\newblock \url{http://www.sagemath.org}.

\bibitem{SwD-l-adic}
H.~P.~F. Swinnerton-Dyer.
\newblock {\em {\it On $l$-adic representations and congruences for
  coefficients of modular forms}, {\rm in: Modular functions of one variable
  {III}}}, pages 1--55.
\newblock \rm Lecture {N}otes in {M}ath. {\bf 350}. Springer, 1973.

\bibitem{Zi}
C.~Ziegler.
\newblock Jacobi forms of higher degree.
\newblock {\em Abh.\ {M}ath.\ {S}em.\ {U}niv.\ Hamburg}, {\bf 59}:191--224,
  1989.

\end{thebibliography}

\end{document}